\documentclass[a4paper,11pt,twocolumn]{article}
\usepackage[english]{babel}
\usepackage{mathrsfs}
\usepackage{amsfonts}
\usepackage[centertags]{amsmath}
\usepackage{amsthm}
\usepackage{latexsym,amsmath,amssymb}
\usepackage{graphicx,color}
\newcommand{\ds}{\displaystyle}

\newcommand{\N}{\mathbb{N}}
\newcommand{\R}{\mathbb{R}}
\newcommand{\K}{\mathscr{K}}
\newcommand{\HK}{\mathscr{H}}
\newcommand{\eps}{\epsilon}
\newcommand{\gref}{\gamma_{\mathrm{ref}}}

\newtheorem{thm}{Theorem}[section]
\newtheorem{prop}[thm]{Proposition}
\newtheorem{defn}[thm]{Definition}
\newtheorem{lemma}[thm]{Lemma}
\newtheorem{corol}[thm]{Corollary}

\newtheoremstyle{obs}
  {3pt}
  {3pt}
  {}
  {}
  {\bfseries}
  {.}
  {.5em}
  {}
\theoremstyle{obs}
\newtheorem{remark}[thm]{Remark}

\def\qed{\ifvmode\removelastskip\fi
{\unskip\nobreak\hfil\penalty50\hbox{}\nobreak\hfil \hbox{\vrule
height1.2ex width1.2ex}\parfillskip=0pt \finalhyphendemerits=0
\par \smallskip}}

\def\tabaddress#1{{\small\it\begin{tabular}[t]{c}#1
\\[1.2ex]\end{tabular}}}

\title{High order sufficient conditions for tracking}
\author{{\sc M. Barbero-Li\~n\'an}
\thanks{{\bf e}-{\it mail}: mbarbero@mast.queensu.ca} \ \thanks{At present, Department of Mathematics \&
Statistics, Queen's University, Kingston (ON), Canada.} \ , {\sc M.
Sigalotti\thanks{{\bf e}-{\it mail}: mario.sigalotti@inria.fr}}
\\
 \tabaddress{INRIA (Projet CORIDA) and Institut \'Elie Cartan de
Nancy, \\ Universit\'e Nancy 1, BP 239, Vand{\oe}uvre-l\`es-Nancy,
54506 France.} }

\begin{document}

\maketitle

\begin{abstract}
In this paper, we study under which conditions the trajectories of a mechanical
control system
 can track any curve on the configuration
manifold. We focus on systems that can be represented as forced
affine connection control systems and we generalize the sufficient
conditions for tracking known in the literature. The sufficient
conditions are expressed in terms of convex cones of vector fields
defined through  particular brackets of the control vector fields of
the system. The tracking control laws obtained by our constructions
depend on several parameters. By imposing suitable asymptotic
conditions on such parameters, we construct algorithmically
one-parameter tracking control laws. The theory is supported by
examples of control systems associated with elliptic hovercrafts and
ellipsoidal submarines.
\end{abstract}

%
%
%


\section{Introduction}

The tracking problem has gained an increasing interest, mainly
because of its applications to robot manipulators as for instance to
control the position of underwater vehicles \cite{Mario} and
hovercrafts \cite{Hovercraft}. The tracking problem appears when a
particular trajectory has to be followed by a control system, but
there is no control law for the control system that makes this
trajectory admissible. Then, the best that can be expected is to
find a control law, typically oscillatory, that defines a good
enough approximation of the target trajectory.

The mathematical background in the tracking problem includes the
averaging theory \cite{BookAveraging} as explained, for instance, in
\cite{2005BulloAndrewBook}. The averaging techniques transform
differential equations difficult to solve into other differential
equations whose solutions approximate fairly well the solutions to
the first set of equations. This is useful to approximate solutions
to differential equations that depend on time or on parameters.

Differential geometry has provided a suitable framework to study in
an intrinsic way typical mechanical control systems in engineering
as, for instance, underwater submarines, aircraft models,
hovercrafts and so on \cite{2005BulloAndrewBook}. Here, we focus on
forced affine connection control systems and generalize the
sufficient conditions for tracking a trajectory that exist in the
literature from a geometric viewpoint \cite{2005BulloAndrewBook}.
The existent results can be interpreted as first\textendash{}order
sufficient conditions because only the control vector fields and
particular brackets, called symmetric products, between them get
involved in the statement of the sufficient conditions. However, our
conditions need longer symmetric products and so they are said to be
of order higher than two.

In \cite{Mario} it was observed that the tracking is possible for
specific underwater vehicles, even though they do not satisfy the
geometric sufficient conditions known in the literature. That motivates
our research in order to obtain more general geometric sufficient
conditions that ensure the tracking property for a wider range of
control systems.

 The chances to be able to track a target trajectory
are related to some controllability requirement and to the avoidance of
``bad" directions. In an informal way, these ``bad" directions have
to be interpreted as directions that will not make possible to have
the starting point in the interior of the reachable set.
(We refer to \cite{2005BulloAndrewBook} for an accurate description
of the obstructions to controllability in terms of the symmetric
products and of particular vector\textendash{}valued quadratic
forms.) The sufficient conditions for being able to track unfeasible
trajectories are also related to how nonholonomicity allows to
enlarge the set of admissible velocities for the control system. In
this regard, constructions of convex cones \cite{Bressan}, the
above\textendash{}mentioned vector\textendash{}valued quadratic
forms \cite{2005BulloAndrewBook} and some techniques similar to the
ones in \cite{Jorge} have been useful for obtaining the
constructions considered here.

The paper is organized as follows. Section \ref{S2Notation} contains
the necessary background in forced affine connection control systems
and in chronological calculus
\cite{AgrachevBook}. Section \ref{STrack} defines properly the
notion of trackability and reviews the geometric sufficient
conditions in the literature \cite{2005BulloAndrewBook}. Sections
\ref{SgeneralTrack} and \ref{Sparameter} contain the main
contributions of the paper and some examples to justify the utility
of these results.

\section{Notation and preliminaries}\label{S2Notation}

Denote by $\mathbb{N}$ the set of {\it positive} natural numbers
and write  $\N_0$ for $\N\cup\{0\}$.
Fix $n\in\mathbb{N}$.
From now on, $Q$ is a $n$\textendash{}dimensional smooth manifold
and $\mathfrak{X}(Q)$ denotes the set of smooth vector fields on
$Q$. All the vector fields are considered smooth, unless otherwise
stated. Let $\tau_Q\colon TQ \rightarrow Q$ be the canonical tangent
projection, a \textit{vector field $X$ on $Q$ defined along
$\tau_Q$} is a mapping $X\colon TQ \rightarrow TQ$ such that
$\tau_Q\circ X= \tau_Q$.

\subsection{Affine connection control systems}

The trajectories $\gamma:I\to Q$ of a Lagrangian mechanical systems
on a manifold $Q$ are minimizers of the action functional
\[A_L(\gamma)=\int_I L(t,\dot{\gamma}(t)){\rm d}t\]
associated with a Lagrangian function $L\colon \mathbb{R} \times TQ\rightarrow \mathbb{R}$.

 The solutions to this variational problem 
must 
satisfy
the well\textendash{}known Euler\textendash{}Lagrange
equations,
\begin{equation}\label{euler-lagrange}
\frac{{\rm d}}{{\rm d}t}\left( \frac{\partial L}{\partial
v^i}\right)-\frac{\partial L}{\partial q^i}=0,\quad i=1,\dots,n,
\end{equation}
where $(q^i,v^i)$ are local coordinates for $TQ$. Here we consider
controlled Euler\textendash{}Lagrange equations obtained by
modifying the right\textendash{}hand side on the above equation, as
follows:
\[\frac{{\rm d}}{{\rm d}t}\left( \frac{\partial L}{\partial
v^i}\right)-\frac{\partial L}{\partial
q^i}=\sum_{a=1}^ku_aY_a^i,\quad i=1,\dots,n,\] with $u_a\colon I
\rightarrow \mathbb{R}$, $Y_a^i\colon Q \rightarrow \mathbb{R}$.

 When the manifold $Q$ is endowed with the Riemannian structure given by a Riemannian metric
$g$ and the Lagrangian function $L_g(v_q)=\frac{1}{2}g(v_q,v_q)$ is
considered, the solutions to (\ref{euler-lagrange}) turn out to be the
geodesics of the Levi\textendash{}Civita affine connection
$\nabla^g$ associated with the Riemannian metric. (See
\cite{2005BulloAndrewBook} for more details and
for many examples of mechanical control systems that fit in this
description.)

When control forces are added to the geodesic equations we obtain an
affine connection control system
\[\nabla^g_{\dot{\gamma}(t)}\dot{\gamma}(t)=\sum_{a=1}^ku_a(t)Y_a(\gamma(t)),\]
with $Y_a$ being vector fields on $Q$.

The notion of affine connection control system can be extended
without the need of the Levi\textendash{}Civita connection.

\begin{defn} An \textbf{affine connection}
is a mapping
\[\begin{array}{rcl}
\nabla\colon\mathfrak{X}(Q)\times \mathfrak{X}(Q)&
\longrightarrow & \mathfrak{X}(Q)\\
(X,Y) & \longmapsto & \nabla(X,Y)=\nabla_XY,
\end{array}\] satisfying the following
properties:
\begin{enumerate}
\item $\nabla$ is $\mathbb{R}$\textendash{}linear on $X$ and on $Y$;
\item $\nabla_{fX}Y=f\nabla_XY$ for every $f\in {\mathcal C}^{\infty}(Q)$;
\item $\nabla_XfY=f\nabla_XY+\left(Xf \right)Y$, for every $f\in {\mathcal C}^{\infty}(Q)$. (Here $Xf$ denotes the derivative
of $f$ in the direction $X$.)
\end{enumerate}
\end{defn}

 The mapping $\nabla_XY$ is called the \textit{covariant derivative of $Y$
 with respect to $X$}. Given local coordinates $(q^i)$ on $Q$, the
 \textit{Christoffel symbols for the affine connection} in these
 coordinates are given by
 \[\ds{\nabla_{\frac{\partial}{\partial q^j}}\frac{\partial}{\partial q^r}=\sum_{i=1}^n
 \Gamma^i_{jr}\frac{\partial}{\partial q^i}.}\]
 From the properties of the affine connection, we have
 \[\ds{\nabla_XY=\sum_{i,j,r=1}^n\left(X^j\frac{\partial Y^i}{\partial q^j}+\Gamma^i_{jr}X^jY^r
 \right)\frac{\partial}{\partial q^i}},\]
 where $X=\sum_{i=1}^nX^i\partial/\partial q^i$ and $Y=\sum_{i=1}^nY^i\partial/\partial
 q^i$.

\begin{defn}\label{Def-FACCS} A \textbf{forced affine connection control system
(FACCS)} is a control mechanical system given by $\Sigma=(Q,\nabla,
Y, \mathscr{Y},
 U)$ where
\begin{itemize}
\item $Q$ is
a smooth $n$\textendash{}dimensional manifold called the
\textit{configuration manifold},
\item $Y$ is a time-dependent vector field along the projection $\tau_Q\colon TQ\rightarrow Q$, measurable and bounded with respect to
the time and affine with respect to the velocities,
\item $\mathscr{Y}$ is a set of $k$
control vector fields on $Q$, and \item $U\subseteq \mathbb{R}^k$.
\end{itemize}
A trajectory $\gamma\colon I\subset \mathbb{R} \rightarrow Q$ is
\textbf{admissible for $\mathbf{\Sigma}$} if $\dot \gamma\colon I
\rightarrow TQ$ is absolutely continuous and there exists a
measurable and bounded control $u\colon I \rightarrow U$ such that
the dynamical equations of the control system~$\Sigma$
\begin{equation}\label{nabla}\nabla_{\dot{\gamma}(t)}\dot{\gamma}(t)=Y(t,\dot{\gamma}(t))+\sum_{a=1}^ku_a(t)Y_a(\gamma(t)),
\end{equation}
are fulfilled (almost everywhere).
\end{defn}

The vector field $Y$ 
includes all the non-controlled external forces; e.g., the potential
and the non\textendash{}potential forces. The assumption that $Y$
is affine with respect to the velocities means that, in every
local system of coordinates $(q^i,v^i)$ on $TQ$, $Y$ can be written
as
$$Y(t,v_q)=Y_0(t,q)+\sum_{i=1}^n v^i Y^i(t,q).$$

Equation $(\ref{nabla})$ can be rewritten as a
first\textendash{}order control\textendash{}affine system on $TQ$,
\begin{equation}
\label{first}
\dot{\Upsilon}(t)=Z(\Upsilon(t))+Y^V(t,\Upsilon(t))+\sum_{a=1}^ku_a(t)Y^V_a(\Upsilon(t)),
\end{equation}
where $\Upsilon\colon I \rightarrow TQ$ is such that  $\tau_Q \circ
\Upsilon=\gamma$, $Z$ is the geodesic spray associated to the affine
connection on $Q$ and $Y^V_a$ denotes the vertical lift of the
vector field $Y_a$ (see \cite{AbrahamMarsden} for more details).

Apart from the usual Lie bracket that provides $\mathfrak{X}(Q)$
with a Lie algebra structure, the following product of elements in
$\mathfrak{X}(Q)$ associated to $\nabla$ can be introduced.
\begin{defn} The \textbf{symmetric product} is the map
\begin{eqnarray*}\label{sym} \langle \cdot \colon \cdot \rangle\colon
\mathfrak{X}(Q) \times  \mathfrak{X}(Q) & \longrightarrow &
\mathfrak{X}(Q) \\
(X,Y) & \longmapsto & \nabla_XY +\nabla_Y X.
\end{eqnarray*}
\end{defn}

It can be proved that
\begin{equation}\label{triple_bracket}
[Y_a^V,[Z,Y_b^V]]=\langle Y_a \colon Y_b
\rangle^V
\end{equation}
(see \cite{2005BulloAndrewBook}).

\subsection{Chronological calculus}

We recall in this section some notion of chronological calculus,
which
is used later as a tool in the study of the asymptotic behavior  of endpoint
mappings depending on parameters. For a  comprehensive discussion
and for the proofs of all results stated in this section see
\cite{AgrachevBook}. In the sequel all vector fields, autonomous and
non\textendash{}autonomous, are assumed to be complete. The behavior
of non\textendash{}complete vector fields on compact sets can be
studied by considering suitable cut\textendash{}off procedures.

Given a non\textendash{}autonomous vector field $X_\tau$ on some
manifold $M$, where $\tau$ denotes the time variable and the map
$(\tau,q)\mapsto X_\tau(q)$ is assumed to be smooth with respect to
$q$ and measurable bounded with respect to $\tau$, we denote by
$\overrightarrow{{\rm exp}}\int^t_0 X_{\tau}{\rm d}\tau$ the
diffeomorphism of $M$ onto itself corresponding to the flow from
time $0$ to time $t$ of $X_\tau$. Hence, $\overrightarrow{{\rm
exp}}\int^t_0 X_{\tau}{\rm d}\tau(\bar q)$ is the evaluation at time
$t$ of the solution to the non\textendash{}autonomous Cauchy problem
\[\dot q(\tau)=X_\tau(q(\tau)),\quad q(0)= \bar q.\]
Any diffeomorphism $P:M\rightarrow M$ defines an isomorphism ${\rm
Ad}P$ of $\mathfrak{X}(M)$ through the rule
$${\rm Ad}P(Y)(q)=P_*(Y(P^{-1}(q))),$$
where 
$P_*$ denotes the
pushforward by $P$.

If $P=\overrightarrow{{\rm exp}}\int^t_0 X_{\tau}{\rm d}\tau$, then
we write
$${\rm Ad}P=\overrightarrow{{\rm exp}}\int^t_0 {\rm ad}X_{\tau}{\rm d}\tau.$$
If ${\rm ad}^{(m)}X_\tau(Y)=0$ for some $m\in\mathbb{N}$ and every
$\tau$, then
\begin{small}
\begin{equation}\label{nilpotent_exponential}\begin{array}{l}
\overrightarrow{{\rm exp}}\int^t_0 {\rm ad}X_{\tau}{\rm d}\tau(Y)=Y+\int^t_0 [X_{\tau},Y]{\rm d}\tau\\
+\int^t_0 \int_{0}^{\tau_1}[X_{\tau_2},[X_{\tau_1},Y]{\rm d}\tau_2{\rm d}\tau_1+\cdots+\\
\int^t_0 \int_{0}^{\tau_1}\cdots \int_{0}^{\tau_{m-2}} {\rm
ad}X_{\tau_{m-1}}\circ\cdots \circ {\rm ad}X_{\tau_{1}}Y{\rm
d}\tau_{m-1}\cdots {\rm d}\tau_1.
\end{array}
\end{equation}
\end{small}
In the framework of chronological calculus the flow of the sum of
two non\textendash{}autonomous vector fields can be conveniently
represented by the following variation formula
\begin{equation}\begin{array}{c}
 \overrightarrow{{\rm
exp}}\int^t_0(X_{\tau}+Y_{\tau}){\rm d}\tau=
\\ \overrightarrow{{\rm exp}}\int^t_0 X_{\tau}{\rm
d}\tau\circ\overrightarrow{{\rm
exp}}\int^t_0\left(\overrightarrow{{\rm exp}}\int^\tau_0{\rm
ad}X_{s}{\rm d}s\right) Y_{\tau} {\rm d}\tau.
\label{variation_formula}
\end{array}
\end{equation}
Let us recall a useful result for the convergence of flows of
non\textendash{}autonomous vector fields. It states, roughly
speaking, that the flows converge if the vector fields converge in
integral sense. For further results and a discussion on this kind of
convergence from the point of view of ordinary differential
equations and control theory, see \cite{Biles,LiuSussmann}.
\begin{lemma}\cite[Lemma 8.10]{AgrachevBook} Let $Z^j_{\tau}$, $j\in\mathbb{N}$,
and $Z_{\tau}$, $\tau \in [0,t_1]$, be non\textendash{}autonomous
vector fields on $M$, bounded with respect to $\tau$, and let these vector
fields have a compact support. If
\[\int^t_0Z^j_{\tau}{\rm d} \tau \rightarrow \int^t_0Z_{\tau} {\rm
d}\tau, \quad j\rightarrow \infty,\] then
\[\overrightarrow{{\rm exp}}\int^t_0Z^j_{\tau}{\rm d} \tau \rightarrow \overrightarrow{{\rm exp}}\int^t_0Z_{\tau} {\rm
d}\tau, \quad j\rightarrow \infty,\] both convergences being uniform
with respect to $(t,q)\in [0,t_1]\times M$ and uniform with all derivatives
with respect to $q\in M$. \label{ConvIntegral}
\end{lemma}

Another, even more standard, result on the convergence of flows is
the following, that we find useful to state as a separate lemma. Its
difference from the previous one can be described as follows:
if the dependence on time of the $Z^j_{\tau}$ is prescribed independently on $j$ (it would be enough that they converge uniformly), then the uniform convergence of flows 
is guaranteed
 by the uniform convergence of the vector fields,
 even without any knowledge about the convergence of their derivatives with respect to the state variables.
For completeness, we provide a brief proof for it.

\begin{lemma}\label{l-unif-conv}
Let $\lambda_1,\dots,\lambda_m\in L^\infty([0,t_1],\mathbb{R})$. Let
$Z^j_\tau$, $n\in\mathbb{N}$, and $Z_{\tau}$, $\tau \in [0,t_1]$, be
non\textendash{}autonomous vector fields on $M$ of the form
$Z^j_\tau=\lambda_1(\tau)Z^j_1+\cdots+ \lambda_m(\tau)Z^j_m$ and
$Z_\tau=\lambda_1(\tau)Z_1+\cdots+ \lambda_m(\tau)Z_m$, with
$Z^j_i,Z_i\in \mathfrak{X}(M)$ with compact support. If
$Z^j_i\rightarrow Z_i$, $j\rightarrow \infty$, then
\[\overrightarrow{{\rm exp}}\int^t_0Z^j_{\tau}{\rm d} \tau \rightarrow \overrightarrow{{\rm exp}}\int^t_0Z_{\tau} {\rm
d}\tau, \quad j\rightarrow \infty,\]
 both convergences being
uniform (the first with respect to the state $q$ and the second with
respect to $(t,q)$).
\end{lemma}
\begin{proof} Using the variation formula (\ref{variation_formula}) with
$X_\tau=Z_\tau$ and $Y_\tau=Z^j_\tau-Z_\tau$, we get
\begin{equation}
\label{special_variation_formula}
\begin{array}{l} \overrightarrow{{\rm exp}}\int^t_0 Z^j_{\tau}{\rm
d}\tau=\overrightarrow{{\rm exp}}\int^t_0 Z_{\tau}{\rm d}\tau\circ
\\ \overrightarrow{{\rm exp}}\int^t_0\left(\overrightarrow{{\rm
exp}}\int^\tau_0{\rm ad}Z_{s}{\rm d}s\right) (Z^j_{\tau}-Z_\tau)
{\rm d}\tau. \end{array}
\end{equation}
Since $\left( \overrightarrow{{\rm exp}}\int^\tau_0 Z_{s}{\rm
d}s\right)_*$, $\tau\in[0,t_1]$, is a compact family of operators,
then
 the last diffeomorphism in (\ref{special_variation_formula}) converges uniformly to the identity for $t\in[0,t_1]$ as $j$ goes to infinity.
\end{proof}

\section{Tracking problem}\label{STrack}

The need of trackability appears when one tries to follow a
particular trajectory on the configuration manifold, called \textit{reference} or
\textit{target} trajectory, which is not a solution of the FACCS
considered. A trajectory is successfully tracked if there exist
solutions to the FACCS that approximate it arbitrarily well.

Consider any distance ${\rm d}\colon Q\times Q \rightarrow \mathbb{R}$ on $Q$
whose corresponding metric topology coincides with the topology on $Q$.
From now on $I$ will denote the interval $[0,t_1]$, with $t_1>0$.

\begin{defn}
A control system $\Sigma$ is \textbf{trackable} if for every
continuous curve $\gamma\colon I\rightarrow Q$, for every $v\in T_{\gamma(0)}Q$ and for every
strictly
positive tolerance $\epsilon$, there exist a control $u^{\epsilon}\in L^\infty(I,U)$
such that the solution $\xi^{\epsilon}\colon I
\rightarrow Q$  to $\Sigma$ corresponding to  $u^\eps$ and  with
initial condition $\dot{\xi}^\epsilon(0)=v$ verifies
\[{\rm d} (\gamma(t), \xi^{\epsilon}(t))<\epsilon\]
for every time $t\in I$.
\end{defn}
\begin{remark}
Since any continuous 
curve can be uniformly approximated, with arbitrary precision, by a smooth curve having a prescribed tangent vector at its initial point,
then $\Sigma$ is trackable if and only if
every curve $\gamma\colon I\rightarrow Q$ on
$Q$ of class $\mathcal{C}^\infty$ is {\bf trackable for $\Sigma$}, i.e.,
for every  $\epsilon>0$, there exist $u^{\epsilon}\in L^\infty(I,U)$
such that the solution $\xi^{\epsilon}\colon I
\rightarrow Q$  to $\Sigma$ corresponding to  $u^\eps$ and  with
initial condition $\dot{\xi}^\epsilon(0)=\dot\gamma(0)$ verifies
${\rm d} (\gamma(t), \xi^{\epsilon}(t))<\epsilon$
for every $t\in I$.
\end{remark}

In order to give some insights into particular sufficient conditions
for tracking, we are going to review a result in the literature.

\begin{thm}\cite[Theorem 12.26]{2005BulloAndrewBook}\label{ACCStrack}
Let $\Sigma=(Q,\nabla,Y, \mathscr{Y},U)$ be a FACCS such that
\begin{itemize}
\item the distribution generated by $\mathscr{Y}=\{Y_1,\dots,Y_k\}$ is regular, that is, it is a subbundle of $TQ$;
\item $\langle Y_a\colon Y_a\rangle \in {\rm span}_{{\mathcal
C}^{\infty}(Q)}\mathscr{Y}$ for every $a\in \{1,\ldots,k\}$, i.e.,
$\langle Y_a\colon Y_a \rangle =\sum_{b=1}^k\sigma_{ab}Y_b$, where
$\sigma_{ab}\in {\mathcal C}^{\infty}(Q)$;
\item the distribution ${\rm Sym}^{(1)}(\mathscr{Y})$ defined by ${\rm Sym}^{(1)}(\mathscr{Y})_q={\rm
span}_{\mathbb{R}}\mathscr{Y}_q+{\rm span}_{\mathbb{R}}\{\langle
W\colon Z \rangle (q) | W,Z \in \mathscr{Y}\}$ is the entire tangent
bundle $TQ$.
\end{itemize}

Let $\gref\colon I \rightarrow Q$ be a reference trajectory of
class ${\mathcal C}^3$. Define the functions
$u_{\mathrm{ref},a},u_{\mathrm{ref},bc}\colon I \rightarrow
\mathbb{R}$, $a,b,c\in \{1,\ldots,k\}$, $b<c$, as solutions of class
${\mathcal C}^1$ to
\[\begin{array}{l}\nabla_{\dot{\gamma}_{\mathrm{ref}}(t)}\dot{\gamma}_{\mathrm{ref}}(t)-Y(t,\dot{\gamma}_{\mathrm{ref}}(t))
\\= \displaystyle\sum_{a=1}^k
u_{\mathrm{ref},a}(t)Y_a(\gamma_{\mathrm{ref}}(t))
\\+\displaystyle\sum_{b,c=1, b<c}^k u_{\mathrm{ref},bc}(t)\langle
Y_b\colon Y_c \rangle (\gamma_{\mathrm{ref}}(t)).\end{array}\]

Define $u_{\mathrm{slow}}\colon I\times TQ \rightarrow
\mathbb{R}^k$, $u_{\mathrm{osc}}\colon \mathbb{R}\times I \times TQ
\rightarrow \mathbb{R}^k$ by
\begin{small}\begin{equation*}\begin{array}{l} u_{\mathrm{slow},a}(t,v_q)=
u_{\mathrm{ref},a}(t)+\\
\displaystyle\sum_{b=1}^k\left(b-1+\displaystyle\sum_{c=b+1}^k
\frac{(u_{\mathrm{ref},bc}(t))^2}{4}\right)\sigma_{ab}(q),
\\ u_{\mathrm{osc},a}(\tau,t,v_q)=\displaystyle\sum_{c=1}^{a-1}\varphi_{{\rm lo}(c,a)}(\tau)\\ -\frac{1}{2}
\displaystyle\sum_{c=a+1}^ku_{\mathrm{ref},ac}(t) \varphi_{{\rm
lo}(a,c)}(\tau),\end{array}\end{equation*}\end{small}  \\
where ${\rm
lo}(a,b)=\displaystyle\sum_{j=1}^{a-1}(k-j)+(b-a)$ for $(a,b)\in
\{1,\ldots,k\}^2$, $a<b$, and for $i\in \mathbb{N}$
\begin{equation}\label{varphi}\varphi_i\colon \mathbb{R} \rightarrow
\mathbb{R}, \quad t \mapsto \frac{4 \pi i}{T}\cos \left( \frac{2 \pi
i}{T} t \right).\end{equation}

Then $\gref$ is trackable for $\Sigma$ and moreover the solutions
$\xi^\eps$ of $\Sigma$, $\eps>0$, corresponding to the controls
\[\begin{array}{rcl} u^{\epsilon}\colon \mathbb{R}\times TQ & \longrightarrow &
U\subset \mathbb{R}^k \\
(t,v_q) & \longmapsto &
u_{\mathrm{slow}}(t,v_q)+\frac{1}{\epsilon}u_{\mathrm{osc}}\left(\frac{t}{\epsilon},t,v_q\right)
\end{array}\]
with initial condition
$\dot\xi^\eps(0)=\dot\gamma_{\mathrm{ref}}(0)$ are such that ${\rm
d}(\gref(t),\xi^\eps(t))$ tends to zero as $\eps$ goes to zero
uniformly with respect to  $t\in I$.
\end{thm}

\remark Observe that under the hypotheses of Theorem \ref{ACCStrack}
not only the tracking is guaranteed, but also the tracking control
law is given explicitly.

\section{A more general tracking result}\label{SgeneralTrack}
The idea of the following construction is to identify, given the set
$\mathscr{Y}$, a larger set of control vector fields $\K_1$ such
that every trajectory solution of the FACCS obtained by replacing
$\mathrm{span}_{\mathcal{C}^\infty(Q)}\mathscr{Y}$ by $\K_1$ can be
tracked by solutions of the original FACCS $\Sigma$. Repeating the
construction on $\K_1$ we obtain an even larger family $\K_2$ and so
on. If eventually $\K_l(q)=T_q Q$ for every $q\in Q$ for some
$l\in\mathbb{N}$, then we can show that the system is trackable.

In order to generalize the sufficient conditions for tracking
given in Theorem \ref{ACCStrack}, we construct the following
set of vector fields on $Q$:
\begin{equation}\label{setK}\begin{array}{rcl} {\mathscr{K}}_0&=&
\overline{{\rm span}_{\mathcal{C}^\infty(Q)}\mathscr{Y}},\\[1.5mm]
\mathscr{K}_l&=&\overline{\mathscr{K}_{l-1}- {\rm co}\left\{\langle
Z \colon Z \rangle \mid Z\in {\rm L}(\mathscr{K}_{l-1})\right\}},
\end{array}\end{equation} for $l\in \mathbb{N}$, where, for $A\subset \mathfrak{X}(Q)$,
 ${\rm L}(A)=A\cap (-A)$, ${\rm co}(A)$ denotes the convex hull  of $A$, and
 $\overline{A}$ is the closure of $A$
in $\mathfrak{X}(Q)$ with respect to the topology of
the uniform convergence on compact sets.
 For $A\subset
\mathfrak{X}(Q)$ we also write
$$A(q)= \{Y(q)\mid Y\in A\}\subseteq T_qQ.$$

\begin{prop}\label{Kcone} For every $l\in \mathbb{N}_0$, $\K_l$ is a convex cone of
$\mathscr{X}(Q)$. In particular,
for every $q\in Q$, $\K_l(q)$ is a
convex cone with vertex at $0\in T_qQ$.
\end{prop}

\begin{proof} The proposition is proved by induction having in mind that
the set $\K_l$ is a convex cone if it contains all conic
combinations of elements of $\K_l$. Remember that a conic
combination of elements of $\K_l$ is of the form
$\lambda_1W_1+\dots+\lambda_r W_r$ with $\lambda_i \in
\mathbb{R}_{\geq 0}$, $W_i\in \K_l$ for every $i\in
\{1,\dots,r\}$.

First, $\K_0$ is subspace of
$\mathscr{X}(Q)$ and thus it is a convex cone.

The induction step consists of proving that if $\K_i$ is a convex
cone, so is $\K_{i+1}$. First notice that, if $W\in L(\K_i)$ and $\lambda\in \R$, then, by the induction
hypothesis, $\lambda W\in L(\K_i)$ and thus $\lambda^2 \langle W\colon W\rangle$ belongs to $\{\langle Z\colon Z\rangle\mid Z\in L(\K_i)\}$.
Hence, 
$\K_{i+1}$ is the closure of the sum of two convex sets invariant by multiplication by any non-negative
scalar. Therefore, $\K_{i+1}$ is itself a convex cone.
 \end{proof}

Given two functions $\alpha,\beta:[0,T]\to \mathbb{R}$, let
\[\Lambda_{T}(\alpha,\beta)=\frac{1}{2T}\int_0^{T}\left(\int_0^\tau \alpha(s){\rm d}s\!\right)\left(\int_0^\tau \beta(s){\rm d}s\!\right){\rm d}\tau.
\]
We say that a \textit{sequence} of smooth $T$-periodic functions
$\psi_j:\mathbb{R}\to \mathbb{R}$, $j\in\mathbb{N}$, is
$\Lambda_T$\textendash{}\textit{orthonormal and zero-mean} if
$\int_0^T \psi_j(\tau)=0$ for every $j\in\N$ and
$\Lambda_T(\psi_j,\psi_m)=\delta_{jm}$ for $j,m\in\mathbb{N}$, where
 $\delta_{jm}$ denotes the Kronecker delta.

For instance, the sequence defined in (\ref{varphi}), is $\Lambda_T$-orthonormal and zero-mean.

\begin{prop}\label{Averaging} Let $\Sigma=(Q,\nabla,Y,\mathscr{Y},\mathbb{R}^{k})$ be a  FACCS and
fix $v_q\in TQ$, $T>0$, $n_1,\dots,n_k\in\mathbb{N}$,  and
$w=(w_1,\dots,w_k)\in L^\infty(I,\mathbb{R}^{k})$. For every
$\eps>0$ denote by $\xi^\eps:I\rightarrow Q$ the solution
 of
\begin{eqnarray*}\label{EqProp} \nabla_{\dot\xi^\eps(t)}\dot\xi^\eps(t)&=&Y(t,\dot\xi^\eps(t))\\
&+& \sum_{a=1}^{k}\frac{1}{\epsilon}
\psi_{n_a}\left(\frac{t}{\epsilon}\right)w_{a}(t)Y_a(\xi^\eps(t)),\end{eqnarray*}
with initial condition $\dot\xi^\eps(0)=v_q$, where $(\psi_{j})_{j\in\mathbb{N}}$
is a $\Lambda_T$-orthonormal and zero-mean sequence. Let also $\gamma:I\rightarrow Q$ be
the solution
 of
\begin{align}&\nabla_{\dot{\gamma}(t)}\dot{\gamma}(t)=Y(t,\dot{\gamma}(t)) \nonumber \\ & -\displaystyle\sum_{a,b=1}^{k}\delta_{n_a
n_b}w_a(t)w_b(t)\langle Y_a\colon Y_b \rangle(\gamma(t)),
\label{EqProp2} \end{align}  with initial condition
$\dot\gamma(0)=v_q$. Then
there exist $C,\eps_0>0$ such that ${\rm d}(\gamma(t),\xi^\eps(t))<
C \epsilon$ for every $t\in I$ and every $\eps\in(0,\eps_0)$.
\end{prop}

\begin{proof} The proof of this proposition follows from Theorem 9.32 in
\cite{2005BulloAndrewBook} by considering the function
$u^a(t/\eps,t)$ that appears there as $\psi_{n_a}(t/\eps)w_a(t)$.
\end{proof}

\begin{remark}\label{betterT}
Theorem~\ref{ACCStrack} follows directly from
Proposition~\ref{Averaging} (and is obtained as such in
\cite{2005BulloAndrewBook}). Therefore, some of the hypotheses of
Theorem~\ref{ACCStrack} can easily be relaxed: first of all the
distribution generated by $\mathscr{Y}$
need not
 be a subbundle of $TQ$.
Moreover, the hypothesis that $\mathrm{Sym}^{(1)}(\mathscr{Y})=TQ$ can be replaced by the requirement that
$$\nabla_{\dot\gamma_{\mathrm{ref}}(t)}{\dot \gamma_{\mathrm{ref}}(t)}-Y(t,\dot\gamma_{\mathrm{ref}}(t))\in \mathrm{Sym}^{(1)}(\mathscr{Y})_{\gamma_{\mathrm{ref}}(t)}$$
for every $t\in I$.
\end{remark}

\begin{thm}\label{GeneralTrack}
Let $\Sigma=(Q,\nabla,Y,\mathscr{Y},\mathbb{R}^{k})$ be  a FACCS.
Fix a 
reference trajectory $\gref:I\rightarrow Q$ of class
${\mathcal C}^\infty$. If for every $t\in I$ there exists $l\in
\mathbb{N}$ such that $\K_l(\gref(t))=T_{\gref(t)}Q$, then $\gref$
is trackable. Therefore, if for every $q\in Q$ there exists $l\in
\mathbb{N}$ such that $\K_l(q)=T_{q}Q$, then the control system
$\Sigma$ is trackable.
\end{thm}
\begin{proof} First of all, notice that the last part of the statement (the trackability of $\Sigma$)
directly follows from the first one (the trackability of $\gref$), which is proved below.

As the reference trajectory $\gref$ is defined on a
compact set and $\gref$ is continuous, then ${\rm Im}\gref$ is a
closed compact set. By hypothesis, for every $t\in I$ there
exists $l\in \mathbb{N}$ such that $\K_l(\gref(t))=T_{\gref(t)}Q$.
So there exist $n+1$ vector fields in $\K_l$ whose conic
combinations at $\gref(t)$ give the whole tangent space $T_{\gref(t)}Q$. The smoothness of these vector fields guarantees that
this is still true in an open neighborhood $U_t$ of $\gref(t)$. In
this way we construct an open cover of  ${\rm Im}\gref$. As  ${\rm
Im}\gref$ is compact, there exists a finite open subcover of
$\{U_t\}_{t\in I}$ given by $U_{t_1},\dots, U_{t_r}$. For each
$t_i$ there exists a different $l_i\in \mathbb{N}$ such that
$\K_{l_i}(\gref(t))=T_{\gref(t)}Q$ for every $t\in
\gref^{-1}(U_{t_i})$. Then $l=\max\{l_1,\dots,l_r\}$
satisfies $\K_l(\gref(t))=T_{\gref(t)}Q$ for every $t\in I$.

Moreover, there exists a partition of unity
subordinated to the finite open subcover $\{U_{t_i}\}_{i=1,\dots,r}$ (see \cite{Lee}) that
allows us to define a finite set of smooth global vector fields
$Z_a$ in $\K_l$ such that
\begin{equation}\label{parameter}\nabla_{\dot{\gamma}_{\mathrm{ref}}(t)}\dot{\gamma}_{\mathrm{
ref}}(t)-Y(t,\dot{\gamma}_{\mathrm{
ref}}(t))=\displaystyle\sum_{a=1}^{N_l}\lambda_a(t)
Z_a(\gref(t))\end{equation} for $t\in I$, $N_l\in \mathbb{N}$, where
$\lambda_a\colon I \rightarrow [0,+\infty)$ is of class ${\mathcal
C}^\infty$.

We introduce for a curve on $Q$ the notion
of having a \textit{regular
parameterization on $\K_i$}. A curve $\gamma\colon I\rightarrow Q$ 
admits a \textit{regular parameterization  on $\K_i$} if there exist
$N_i\in \mathbb{N}$, $\lambda_1,\dots,\lambda_{N_i}\in {\mathcal
C}(I)$ and $Z_1,\dots,Z_{N_i}\in \K_i$ such that $\lambda_a(t)\geq 0$ for every $t\in I$ and
$a\in \{1,\dots,N_i\}$ and
\[\nabla_{\dot{\gamma}(t)}\dot{\gamma}(t)-Y(t,\dot{\gamma}(t))=
\displaystyle\sum_{a=1}^{N_i}\lambda_a(t) Z_a(\gamma(t)).
\]

Fix $\epsilon>0$ and let $\xi_l=\gref$. Then $\xi_l$ admits a regular
parameterization  on $\K_l$.
We are going to prove by induction that there exists
a finite sequence
$\{\xi_0,\xi_{1},\dots,\xi_l\}$ of curves on $Q$ such that
each $\xi_i$ satisfies $\dot\xi_i(0)=\dot\gamma_{\mathrm{ref}}(0)$, admits a regular
parameterization on $\K_{i}$ and ${\rm
d}(\xi_i(t),\xi_{i+1}(t))<\eps/l$ for every $i=0,\dots,\l-1$ and
every $t\in I$.
The induction step claims that: for $i=0,\dots,l-1$
 if there exists $\xi_{i+1}\colon I
\rightarrow Q$ admitting a regular parameterization on $\K_{i+1}$,
then there exists $\xi_{i}\colon I \rightarrow Q$ with
$\dot{\xi}_{i+1}(0)=\dot{\xi}_{i}(0)$ admitting  a regular
parameterization on $\K_{i}$ and satisfying $ {\rm
d}(\xi_i(t),\xi_{i+1}(t))<\epsilon/ l$ for every $t\in I$.

Let us prove the induction step for $i$. Let
$Z_1,\dots,Z_{N_{i+1}}\in \K_{i+1}$ be the vector fields that
determine the regular parameterization of $\xi_{i+1}$. Here we split
the proof in two steps, considering first the special case:
\begin{enumerate}
\renewcommand{\theenumi}{\arabic{enumi}}
\item \label{case1}
$Z_a=F_a-G_a$ with
$F_a\in \K_i$ and $G_a
\in
{\rm co}\left\{\langle Z \colon Z \rangle \mid Z\in {\rm
L}(\K_i)\right\}$ for every $a=1,\dots,N_{i+1}$,
\end{enumerate}
and then the general case
\begin{enumerate}
\renewcommand{\theenumi}{\arabic{enumi}}
\setcounter{enumi}{1}
\item
\label{case2}
 $Z_a\in
\overline{\K_i-{\rm co}\left\{\langle Z \colon Z \rangle \mid Z\in {\rm
L}(\K_i)\right\}}.$
\end{enumerate}
Let us study case~\ref{case1}: $\xi_{i+1}$ admits the parameterization
\begin{eqnarray}\lefteqn{\nabla_{\dot{\xi}_{i+1}(t)}\dot{\xi}_{i+1}(t)-Y(t,\dot{\xi}_{i+1}(t))=}\nonumber
\\ &&
\displaystyle\sum_{a=1}^{N_{i+1}}\lambda_a(t)(F_a-G_a)(\xi_{i+1}(t)).
\label{FirstSystemStepl}\end{eqnarray}
Each $G_a$ is given by
$\displaystyle\sum_{b=1}^{N_a} \alpha_{a,b} \langle G_{a,b} \colon
G_{a,b}\rangle$ where $G_{a,b}\in {\rm L}(\K_i)$ and
$\alpha_{a,b}\geq 0$.
Then (\ref{FirstSystemStepl}) becomes
\begin{equation}\label{Stepi}\begin{array}{c}\nabla_{\dot{\xi}_{i+1}(t)}\dot{\xi}_{i+1}(t)-Y(t,\dot{\xi}_{i+1}(t))
=\displaystyle\sum_{a=1}^{N_{i+1}}\Big(\lambda_a(t) \\
F_a(\xi_{i+1}(t))
 -\displaystyle\sum_{b=1}^{N_a}\lambda_a(t)\alpha_{a,b}
\langle G_{a,b}\colon
G_{a,b}\rangle(\xi_{i+1}(t))\Big).\end{array}\end{equation} The
dynamics described in (\ref{Stepi}) are of the same form as in
(\ref{EqProp2}) as long as we take
$\mathscr{Y}=\{G_{a,b}\}_{a=1,\dots,N_{i+1}; b=1,\dots,N_a}$, and
$\delta_{n_{ab}n_{a'b'}}=\delta_{aa'}\delta_{bb'}$, $
w_{a,b}=\sqrt{\lambda_a\alpha_{a,b}}$.

Then by Proposition \ref{Averaging} the solutions to
(\ref{Stepi}) can be approximated by solutions to
\begin{equation}\label{xieps1}\begin{array}{c}
\nabla_{\dot{\xi}_i^{\eps_i}(t)}\dot{\xi}_i^{\eps_i}(t)-Y(t,\dot{\xi}^{\eps_i}_{i}(t))=
\displaystyle\sum_{a=1}^{N_{i+1}}\Big(\lambda_a(t)F_a(\xi_i^{\eps_i}(t))
\\ +\displaystyle\sum_{b=1}^{N_a}
\frac{1}{\eps_i}\psi_{n_{ab}}\left(\frac{t}{\eps_i}\right)w_{a,b}(t)G_{a,b}(\xi_i^{\eps_i}(t))\Big)\end{array}
\end{equation}
being $(a,b)\mapsto n_{ab}$ any injective map from
$\mathbb{N}\times\mathbb{N}$ to $\mathbb{N}$
and $(\psi_j)_{j\in\N}$ a $\Lambda_T$-orthonormal and zero-mean sequence.
More precisely, there exist
$C_i,\eps_{i,0}>0$ such that ${\rm
d}(\xi_i^{\eps_i}(t),\xi_{i+1}(t))<C_i\eps_i$ for every $t\in
I$, $\eps_i\in (0,\eps_{i,0})$. In particular,
we can choose $\eps_{i}$ such that $C_i\eps_i<\eps/l$.

The finite linear combination of elements in
$\K_i(\xi_i^{\eps_i}(t))$ for every $t\in I$ on the
right\textendash{}hand side of (\ref{xieps1}) does not necessarily
satisfy the non-negativeness of the coefficients. However,
$G_{a,b}\in {\rm L}(\K_i)$, so $-G_{a,b}\in {\rm L}(\K_i)$. Then we
can rewrite the coefficients of $G_{a,b}$ as follows:
\[\begin{array}{c}\psi_{n_{ab}}\left(\frac{t}{\eps_i}\right)w_{a,b}(t)G_{a,b}=
\max\left\{0,\psi_{n_{ab}}\left(\frac{t}{\eps_i}\right)\right\}\\
w_{a,b}(t)G_{a,b}
+\max\left\{0,-\psi_{n_{ab}}\left(\frac{t}{\eps_i}\right)\right\}w_{a,b}(t)(-G_{a,b}).\end{array}\]
Thus all the coefficients are continuous and non-negative. We can
conclude that $\xi_i^{\eps_i}$ admits a regular parameterization on $\K_{i}$.
We define then $\xi_i=\xi_i^{\eps_i}$ and the induction step has
been proved for $i$ in the case \ref{case1}.

Let us turn to case~\ref{case2}.
 We recall that the closure appearing in (\ref{setK}) is considered with
respect to the topology of the uniform convergence on compact sets.
Then there exist two sequences $F_a^j\in \K_i$ and  $G_a^j \in {\rm
co}\left\{\langle Z \colon Z \rangle \mid Z\in {\rm
L}(\K_i)\right\}$ such that $F_a^j-G_a^j$ converges uniformly to
$Z_a$ on a neighborhood of the curve $\xi_{i+1}$. For every $j\in
\mathbb{N}$ let $\xi_{i+1}^j\colon I \rightarrow Q$ be the solution
to
\begin{equation}\label{FirstSystemSteplN}\begin{array}{c}\nabla_{\dot{\xi}_{i+1}^j(t)}\dot{\xi}_{i+1}^j(t)-Y(t,\dot{\xi}^j_{i+1}(t))
\\=
\displaystyle\sum_{a=1}^{N_{i+1}}\lambda_a(t)(F_a^j-G_a^j)(\xi_{i+1}^j(t))\end{array}\end{equation}
satisfying $\dot{\xi}_{i+1}^j(0)=\dot{\xi}_{i+1}(0)$.
Then, thanks
to Lemma~\ref{l-unif-conv}, $\xi_{i+1}^j$ converges to $\xi_{i+1}$
uniformly on $I$
as $j$ tends to infinity. Take ${\bar \jmath}$ large enough such that 
\begin{equation*}\label{DNaprox}{\rm
d}(\xi_{i+1}^{\bar \jmath}(t),\xi_{i+1}(t))<\frac{\eps}{2l}\end{equation*} for
every $t\in I$.
We can apply to
$\xi_{i+1}^{\bar \jmath}$ the same reasoning as in the first case. Then
solutions to (\ref{FirstSystemSteplN}) are approximated by solutions
to 
(\ref{xieps1}) replacing $\xi_i^{\eps_i}$ by $\xi_i^{{\bar \jmath},\eps_i}$, $F_a$ by $F_a^{\bar\jmath}$  and
$G_{a,b}$ by $G^{\bar \jmath}_{a,b}$.
In other words, there exist $C_i,\eps_{i,0}>0$ such that ${\rm
d}(\xi_i^{{\bar \jmath},\eps_i}(t),\xi_{i+1}^{\bar \jmath}(t))<C_i\eps_i$ for every $t\in
I$, $\eps_i\in (0,\eps_{i,0})$. Again, $\eps_{i}$ can be chosen
in such a way that $C_i\eps_i<\frac\eps{2l}$ and we define
$\xi_i=\xi^{{\bar \jmath},\eps_i}_{i}$. Thus,
\begin{equation*}\begin{array}{rcl}{\rm
d}(\xi_i(t),\xi_{i+1}(t)) &\leq &{\rm
d}(\xi_i^{{\bar \jmath},\eps_i}(t),\xi_{i+1}^{\bar \jmath}(t))\\ &+& {\rm
d}(\xi_{i+1}^{\bar \jmath}(t),\xi_{i+1}(t))<\frac{\eps}{l}\end{array}\end{equation*}
and $\xi_i$ admits a regular parameterization on $\K_{i}$. Hence,
the induction step has been proved for $i$.

After the induction, we end up with a curve $\xi_0$ on $Q$ admitting a
regular
parameterization on $\K_{0}$ and such that
\[{\rm d}(\xi_0(t),\gref(t))\leq \displaystyle\sum_{i=0}^{l-1} {\rm
d}(\xi_i(t),\xi_{i+1}(t)) < l\frac{\eps}{l}=\eps\] for every $t\in
I$. (Recall that $\xi_l=\gref$.)
Moreover, by compactness of $I$, \[\bar\eps=
\eps-\max_{t\in I} {\rm d}(\xi_0(t),\gref(t))>0.\]
As $\K_0=\overline{{\rm span}_{{\mathcal
C}^{\infty}(Q)}\mathscr{Y}}$, we conclude from
Lemma~\ref{l-unif-conv} that there exist $\theta_{a,j}\in
\mathcal{C}^\infty(Q)$, $a\in\{1,\dots,N_0\}$, $j\in\{1,\dots,k\}$,
such that the solution to
\begin{align*}\nabla_{\dot{\xi}(t)}\dot{\xi}(t)&-Y(t,\dot{\xi}(t))\\ & =\displaystyle\sum_{a=1}^{N_0}
\sum_{j=1}^{k}\lambda_a(t)\theta_{a,j}({\xi}(t))Y_j(\xi(t))\end{align*}
with initial condition
$\dot\xi(0)=\dot{\xi}_0(0)=\dot{\gamma}_{\mathrm{ref}}(0)$ satisfies
${\rm d}(\xi(t),\xi_0(t))<\bar\eps$ for every $t\in I$. Thus, ${\rm d}(\xi(t),\gref(t))<\eps$ for every $t\in I$.

Since 
 $\xi$ is an admissible trajectory for $\Sigma$, we conclude that
$\gref$ is trackable for $\Sigma$ with the tracking control law
given by $u_j(t)=\sum_{a=1}^{N_0} \lambda_a(t)\theta_{a,j}({\xi}(t))$,
$j=1,\dots,k$.
\end{proof}

\begin{corol}\label{CorolH} Let $\Sigma=(Q,\nabla,Y,\mathscr{Y},\mathbb{R}^k)$ be a
FACCS. Define the following set of vector fields on $Q$ for $l\in
\mathbb{N}$:
\begin{equation}\label{setH} \begin{array}{ll} \HK_0=& {\rm span}_{\mathcal{C}^\infty(Q)}\mathscr{Y},\\
\HK_l=&\HK_{l-1}- {\rm co}\left\{\langle Z \colon Z \rangle \mid
Z\in {\rm L}(\HK_{l-1})\right\}.
\end{array}\end{equation}
Fix a smooth reference trajectory $\gref:I\rightarrow Q$. If for
every $t\in I$ there exists $l\in \mathbb{N}$ such that
$\HK_l(\gref(t))=T_{\gref(t)}Q$, then $\gref$ is trackable.
Therefore, if for every $q\in Q$ there exists $l\in \mathbb{N}$ such
that $\HK_l(q)=T_{q}Q$, then the control system $\Sigma$ is
trackable.
\end{corol}
\begin{proof} Let us prove by induction that
\[\HK_i\subseteq \K_i \quad i\in \mathbb{N}_0.\]
It is trivial that $\HK_0\subseteq \K_0$ (see (\ref{setK}) and
(\ref{setH})).

The claim now is that if $\HK_i\subseteq \K_i$, then
$\HK_{i+1}\subseteq \K_{i+1}$.  By definition, an element in
$\HK_{i+1}$ is of the form $F-G$ with $F\in \HK_i$ and $G\in {\rm
co}\left\{\langle Z \colon Z \rangle \mid Z\in {\rm
L}(\HK_i)\right\}$.
Since
$\HK_i\subseteq \K_i$, then $F\in
\K_i$ and ${\rm L}(\HK_i)\subseteq {\rm L}(\K_i)$. So $G \in {\rm
co}\left\{\langle Z \colon Z \rangle \mid Z\in {\rm
L}(\K_i)\right\}$. We can conclude that $\HK_{i+1}\subseteq
\K_{i+1}$.

By hypotheses, for every $t\in I$ there exists $l\in \mathbb{N}$
such that $\HK_l(\gref(t))=T_{\gref(t)}Q$. As $\HK_l\subseteq \K_l$,
we have $\K_l(\gref(t))=T_{\gref(t)}Q$.
The hypotheses of Theorem~\ref{GeneralTrack} are satisfied, so the result holds.
 \end{proof}

\begin{remark}\label{Hcase1}
If one follows the proof of Theorem~\ref{GeneralTrack} under the stronger hypotheses
of Corollary~\ref{CorolH}, then each step of the induction procedure
is of the type considered in case~\ref{case1}.
This will be important in the next section, where
we will turn such procedure in an algorithmic construction.
\end{remark}

\begin{corol}\label{CorolSym1} Let $\Sigma=(Q,\nabla,Y,\mathscr{Y},\mathbb{R}^k)$ be a FACCS. Define the following sets of
vector fields for $l\in \mathbb{N}$,
\begin{align} \mathscr{Z}_0=& \mathscr{Y},\nonumber\\
\mathscr{Z}_l=& \mathscr{Z}_{l-1} \cup \{\langle Z_a\colon
Z_b\rangle \mid Z_a,Z_b\in \mathscr{Z}_{l-1}\}.\label{setZ}
\end{align}
If there exists $l\in \mathbb{N}$ such that ${\rm
span}_{\mathbb{R}}\mathscr{Z}_l(q)=T_qQ$ for all $q\in Q$ and for
each $i\in\{0, \ldots, l-1\}$, for each $Z\in \mathscr{Z}_i$,
$\langle Z\colon Z\rangle\in {\rm
span}_{\mathcal{C}^\infty(Q)}\mathscr{Z}_i$, then the system $\Sigma$ is
trackable.
\end{corol}
\begin{proof} Let us prove by induction that \[{\rm span}_{{\mathcal
C}^\infty(Q)}\mathscr{Z}_i\subseteq \K_i, \quad i\in \mathbb{N}_0.\]
Once this inclusion is proved 
Theorem~\ref{GeneralTrack} guarantees the trackability of the system.

 It is
trivial by definition that ${\rm span}_{{\mathcal C}^\infty(Q)}
\mathscr{Z}_0\subseteq \K_0$.

Assume that ${\rm span}_{{\mathcal
C}^\infty(Q)}\mathscr{Z}_i\subseteq \K_i$ and let us prove the
inclusion for $i$. Since $\K_{i+1}$ is a convex cone by Proposition
\ref{Kcone} and ${\rm span}_{{\mathcal
C}^\infty(Q)}\mathscr{Z}_i\subset \K_{i+1}$, it is enough to prove
that $\alpha \langle Z_a \colon Z_b \rangle$ belongs to $\K_{i+1}$
for $Z_a,Z_b\in \mathscr{Z}_{i}$ and $\alpha\in {\mathcal
C}^\infty(Q)$. Thanks to (\ref{triple_bracket})
and  to the hypotheses on the symmetric products of elements of
$\mathscr{Z}_{i}$,
\begin{eqnarray*}
\langle \alpha Z_a \colon Z_b\rangle\in \alpha \langle Z_a \colon Z_b\rangle+{\rm span}_{{\mathcal
C}^\infty(Q)}\mathscr{Z}_{i},\\
\langle \alpha Z_a \colon \alpha Z_a\rangle,\langle  Z_b \colon
Z_b\rangle\in {\rm span}_{{\mathcal C}^\infty(Q)}\mathscr{Z}_{i}.
\end{eqnarray*}
Hence the symmetric product
\[\langle \alpha Z_a-Z_b \colon \alpha Z_a -Z_b\rangle\]
belongs to $ -2\alpha \langle Z_a \colon Z_b \rangle+{\rm
span}_{{\mathcal C}^\infty(Q)}\mathscr{Z}_{i}$. Thus, $\alpha\langle
Z_a \colon Z_b \rangle$ belongs to ${\rm span}_{{\mathcal
C}^\infty(Q)}\mathscr{Z}_{i}-\frac12\langle \alpha Z_a-Z_b \colon
\alpha Z_a -Z_b\rangle$, which is contained in $\K_{i+1}$.
\end{proof}

\remark The proof above actually shows that, under the assumptions of
Corollary~\ref{CorolSym1},
 ${\rm span}_{{\mathcal
C}^\infty(Q)}\mathscr{Z}_i\subseteq \HK_i$. It is easy to check that, in addition,
$\HK_i={\rm span}_{{\mathcal
C}^\infty(Q)}\mathscr{Z}_i$.

\remark The main interest of Corollary \ref{CorolSym1}
is that its hypotheses are formulated in terms of a finite set of vector fields, in contrast
with the infinite family of vector fields considered in
Theorem~\ref{GeneralTrack} and Corollary~\ref{CorolH}.

\subsection{Examples}\label{exps}

Let us consider some examples of mechanical systems for which the above results
guarantee the trackability, but Theorem \ref{ACCStrack} could not
guarantee it.

\subsubsection{Hovercraft}

Consider an elliptic hovercraft moving 
on the surface of a
fluid,
identified with
$\mathbb{R}^2$. The configuration manifold is $Q=S^1\times
\mathbb{R}^2$ with local coordinates $(\theta,x_1,x_2)$ where
$\theta$ is the attitude and $(x_1,x_2)$ is the position of the
center of symmetry of the hovercraft. Let $\omega$ and $(v_1,v_2)$ be
the standard angular and linear velocity, respectively, of the
hovercraft with respect to a body\textendash{}fixed coordinate frame
attached at the center of symmetry of the body and whose axes coincide with those of the ellipse. Assume that the center of mass
according to that body\textendash{}fixed coordinate frame is on the
horizontal axis and is different from the center of symmetry. Then 
the added inertia
matrix is the following $3\times 3$ symmetric matrix:
\begin{equation*}{\mathcal
M}= \left(\begin{array}{r|lc} a & 0 & c \\ \hline 0 & e & 0 \\ c & 0
& e
 \end{array}\right),\end{equation*}
 with $a,c,e>0$ (see
\cite{Lamb,Milne} for more details). Denote the corresponding
impulse vector by  $(\Pi,P_1,P_2)$ that is related to the velocities
through the inertia matrix ${\mathcal M}$ as follows
\[\begin{pmatrix} \Pi \\ P_1 \\ P_2
 \end{pmatrix} =
{\mathcal M}\begin{pmatrix}  \omega \\ v_1 \\ v_2 \end{pmatrix} .\] The dynamics of the systems governed by the Kirchhoff
equations in dimension 2 with two controls 
are
\[\begin{pmatrix}  \dot\theta \\ \dot{x} \\ \dot{y} \end{pmatrix} =
\begin{pmatrix} 1 & 0 & 0 \\ 0 & \cos \theta & -\sin \theta  \\ 0 &
\sin \theta & \cos \theta
 \end{pmatrix}\begin{pmatrix} \omega \\ x \\ y
 \end{pmatrix},\]\[\begin{pmatrix} \dot{\omega} \\ \dot{v}_1 \\ \dot{v}_2 \end{pmatrix} = {\mathcal M}^{-1} \begin{pmatrix}
 P \cdot v^{\perp} \\ \omega P^{\perp} \end{pmatrix} + \begin{pmatrix} u_1 \\ 0 \\ u_2\end{pmatrix}\]
where $w^{\perp}=(-w_2,w_1)$ denotes the rotation by 
$\pi/2$
of a vector
$w=(w_1,w_2)$ in $\mathbb{R}^2$.
The control vector fields 
are
$Y_1=(1,0,0)$ and $Y_2=(0,0,1)$. They correspond to
the external torque, usually called yaw, and the external force,
usually called surge, applied to the body.
Let
$\mathscr{Y}=\{Y_1,Y_2\}$.

The drift term $Z$ appearing in (\ref{first}) is, for the mechaincal
system considered here, the vector field corresponding to the
uncontrolled Kirchhoff's equations (notice that $Y=0$). Using
(\ref{triple_bracket}) we can compute $\langle Y_1\colon
Y_1\rangle=(0,2c/e,0)$. The sufficient conditions for tracking given
by Theorem \ref{ACCStrack} are not satisfied because $\langle
Y_1\colon Y_1\rangle\notin {\rm span}_{{\mathcal C}^{\infty}(Q)}
\{Y_1,Y_2\}$.

 However, due to Corollary \ref{CorolH} tracking is possible because
\[\HK_1(q)=T_qQ \; \forall \, q\in Q.\]
Indeed, for $w_1,w_2\in {\mathcal C}^\infty(Q)$,
\begin{align*}&\langle w_1Y_1+w_2Y_2\colon w_1Y_1+w_2Y_2
\rangle= \\ &\frac{1}{e}(0,2cw_1^2+2ew_1w_2,0)+ \stackrel{\in \;
{\rm span}_{{\mathcal
C}^{\infty}(Q)}\mathscr{Y}}{\overbrace{2\sum_{i,j=1}^2w_iY_i(w_j)Y_j}}\,,\end{align*}
where $Y_i(w_j)$ denotes the Lie derivative of $w_j$ with respect to
$Y_i$. In particular, taking as $w_1$ any nonzero constant function,
we get
\[{\rm span}_{{\mathcal C}^{\infty}(Q)}(0,1,0) \subset \HK_0- {\rm co}\left\{\langle Z \colon Z \rangle \mid Z\in
{\rm L}(\HK_0)\right\}.\]
So we conclude that $\HK_1(q)=T_qQ$ for
all $q\in Q$.

\subsubsection{Submarine}\label{SExSubmarine}

Let us apply Corollary \ref{CorolSym1} to determine the trackability of
a particular control system describing the motion of a submarine.
The system corresponds to the case $\gamma=0$  considered in
 \cite{Mario}.
It models a neutrally buoyant ellipsoid vehicle immersed in a
infinite volume fluid that is inviscid, incompressible and whose
motion is irrotational. The dynamics are obtained through Kirchhoff
equations \cite{Lamb} and have a particularly simple form due to
some symmetry assumption on the distribution of mass (see
\cite{Mario} for details and also \cite{munnier}
for general overview of control 
motion in a potential
fluid).

Consider the coordinates $(\omega,v)$ for the angular and linear
velocity of the ellipsoid with respect to a body\textendash{}fixed
coordinate frame. Then the impulse $(\Pi,P)$ of the system is given
by
\[ \begin{pmatrix} \Pi \\ P \end{pmatrix} = {\mathcal M}
\begin{pmatrix} \omega \\ v \end{pmatrix}  \]
where, under the symmetry assumptions mentioned above,
$$
{\mathcal M}=\mathrm{diag}(J_1,J_1,J_3,M_1,M_2,M_3)
$$
with $M_1\neq M_2$,
where
$\mathrm{diag}(J_1,J_1,J_3)$ is the usual inertia matrix and $M_1,M_2,M_3$ take into account the mass of the submarine and the added masses due to the action of the fluid.

The configuration manifold $Q$ for this problem is the Special
Euclidean group or the group of rigid motions $SE(3)$, which is
homeomorphic to $\mathbb{R}^3\times SO(3)$. Let $(r,A)\in SE(3)$ be
the position and the attitude of the ellipsoid.
Denote by $S\colon \mathbb{R}^3\rightarrow \mathfrak{so}(3)$ the
linear bijection between $\mathbb{R}^3$ and the linear algebra
$\mathfrak{so}(3)$ of $SO(3)$ such that
\[S(x_1,x_2,x_3)=\begin{pmatrix}0 & -x_3 & x_2 \\ x_3 & 0 & -x_1 \\ -x_2 & x_1 & 0
\end{pmatrix}\]
 The dynamics
of the controlled system are given by
\begin{equation}\label{eq-below}
\ds{\frac{{\rm d}r}{{\rm d} t}= Av,\quad \frac{{\rm d}A}{{\rm d} t}=
A S(\omega),}\end{equation}
and
\begin{equation}\label{eq-KirchSubm}
\ds{\frac{{\rm d} \Pi}{{\rm d}t}} = \Pi \times \omega + P \times v+
\begin{pmatrix} u_1 \\ u_2 \\ 0 \end{pmatrix}, \quad
\ds{\frac{{\rm d} P}{{\rm d}t}}  = P \times \omega+
\begin{pmatrix} 0 \\ 0 \\ u_3 \end{pmatrix}.\end{equation}
The control vector fields are $Y_1=\partial /\partial \Pi_1$,
$Y_2=\partial /\partial \Pi_2$ and $Y_3=\partial /\partial P_3$.
They correspond to a linear acceleration along one of the three axes of the submarine and to two angular accelerations around the other two axes.

Due to Theorem~1.2 in \cite{Mario}, we know that system
(\ref{eq-below})--(\ref{eq-KirchSubm}) is trackable.

This case cannot be recovered from Theorem~\ref{ACCStrack} because
\[{\rm Sym}^{(1)}(\mathscr{Y})\neq TQ.\]
Indeed, it can be computed that
 \begin{align*}
 &\langle Y_1\colon
 Y_2\rangle=0, \langle Y_2\colon Y_3\rangle=\frac{M_3}{M_1}\frac{\partial}{\partial
P_1}, \\ &
 \langle Y_1\colon Y_3\rangle=-\frac{M_3}{M_2}\frac{\partial}{\partial
 P_2}.
 \end{align*}

However, this case is covered by Corollary~\ref{CorolSym1}
with $l=2$, because
 \begin{align*}&0=\langle Y_1\colon Y_1\rangle=\langle Y_2\colon Y_2\rangle=\langle Y_3\colon Y_3\rangle=\\ &=\langle \langle Y_2\colon Y_3\rangle\colon \langle Y_2\colon Y_3\rangle\rangle=\langle \langle Y_1\colon Y_3\rangle\colon \langle Y_1\colon Y_3\rangle\rangle
, \\ &
  \langle \langle Y_2\colon Y_3 \rangle \colon \langle Y_1 \colon  Y_3 \rangle \rangle=-\frac{M_3^2}{J_3}\left(\frac{1}{M_1}-
  \frac{1}{M_2}\right)\frac{\partial}{\partial \Pi_3}.\end{align*}
Thus, $\langle Z\colon Z\rangle=0$ for every $Z\in \mathscr{Z}_2$
and
 $T_qQ={\rm
span}_{\mathbb{R}}\mathscr{Z}_2(q)={\rm span}_{\mathbb{R}}\{Y_1(q),Y_2(q), Y_3(q),\langle Y_2\colon Y_3\rangle(q),$ $\langle Y_1\colon Y_3\rangle(q),$
 $\langle \langle Y_2\colon Y_3 \rangle \colon \langle Y_1 \colon  Y_3 \rangle
 \rangle(q)\}$ for all $q\in Q$.

The model studied in \cite{Mario} can therefore be handled with the techniques proposed here. In particular, we can obtain for it one\textendash{}parameter tracking control laws, as explained in the next section.

\section{One\textendash{}parameter tracking control laws}\label{Sparameter}
\newcommand{\lo}{j}

The aim of this section is to provide an algorithmic implementation
of the results obtained in the previous one about the existence of
controls yielding tracking. This will be done separately under the
hypotheses of Corollaries~\ref{CorolH} and \ref{CorolSym1}, using two different algorithms. The
first one is based on the procedure proposed in the
proof of Theorem~\ref{GeneralTrack}, while the second one exploits the construction proposed in \cite{2005BulloAndrewBook} and
recalled in Theorem~\ref{ACCStrack}.

In both cases we will consider a reference trajectory ${\gamma}_{\mathrm{ref}}:I\to Q$, which is assumed to be of class ${\mathcal{C}}^\infty$.

A simple, albeit crucial, fact that will be used several times in
the following sections is stated in the lemma below.

\begin{lemma}\label{simple_fact}
If $f\colon \mathbb{R}\times I\to\R$, $(\tau,s)\mapsto f(\tau,s)$, is smooth on $\mathbb{R}\times I$ and $T$-periodic with respect to 
$\tau$, then \begin{eqnarray*} \int_0^t f(s/\hat\eps,s){\rm
d}s&=&\int_0^t \bar f(s)ds+O(\hat\eps\|f\|_\infty)\\
&+&O(\hat\eps\|\partial_2 f\|_\infty)\end{eqnarray*}
for $\hat\eps$ close to zero,  where $\bar f(s)=(1/T)\int_0^T f(\tau,s){\rm d}\tau$
 and $\partial_2$ denotes the partial derivative with respect to the second variable.
\end{lemma}

\subsection{Case $\HK$}\label{SCaseH}

As noticed in Remark~\ref{Hcase1}, the hypotheses
of Corollary~\ref{CorolH}
guarantee that every step of the induction argument proposed in the
proof of Theorem~\ref{GeneralTrack} falls
in the framework of case~\ref{case1} (see page~\pageref{case1}).
Hence, starting from a parameterization
\[\nabla_{\dot{\gamma}_{\mathrm{ref}}(t)}\dot{\gamma}_{\mathrm{ref}}(t)-Y(t,\dot{\gamma}_{\mathrm{ref}}(t))
=\sum_{a=1}^{N_l}\lambda_a(t) Z_a^l(\gref(t))\] of $\gref\colon I
\rightarrow Q$, with $Z_a^l\in \HK_l$ and $\lambda_a$ smooth and
non-negative on $I$ for every $a=1,\dots,N_l$, we can construct
algorithmically a $l$-parameter family of admissible trajectories
$\xi^{\eps_1,\dots,\eps_l}$ of $\Sigma$
with $\eps_1,\dots,\eps_l>0$ such that $\xi_l=\gref$, the uniform
limit
\begin{equation}\label{cascade-limits}
\xi_{i}^{\eps_1,\dots,\eps_{l-i}}=\lim_{\eps_l\rightarrow
0}\lim_{\eps_{l-1}\rightarrow 0}\cdots \lim_{\eps_{l-i+1}\rightarrow
0}\xi^{\eps_1,\dots,\eps_l}
\end{equation}
exists for every $i=1,\dots,l$ and every $\eps_1,\dots,\eps_{l-i}>0$
and satisfies
\begin{align}\nabla_{\dot{\xi}_{i}^{\eps_1,\dots,\eps_{l-i}}(t)}\dot{\xi}_{i}^{\eps_1,\dots,\eps_{l-i}}(t)
&-Y(t,\dot{\xi}_{i}^{\eps_1,\dots,\eps_{l-i}}(t)) \nonumber \\ & \in
\HK_{i}({\xi}_{i}^{\eps_1,\dots,\eps_{l-i}}(t)). \label{EqParamCaseH}
\end{align}
We also  write
$\xi_0^{\eps_1,\dots,\eps_l}$ for $\xi^{\eps_1,\dots,\eps_l}$.
It is important to notice that the order of the limits in
(\ref{cascade-limits}) cannot in general be reversed.

Let us recall to which extent the construction is algorithmic. Fix
any injective map $\lo:\mathbb{N}\times\mathbb{N}\rightarrow
\mathbb{N}$. By backward recursion on $i$, if
$\xi_i^{\eps_1,\dots,\eps_{l-i}}$ satisfies
\begin{align}
\nabla_{\dot{\xi}_{i}^{\eps_1,\dots,\eps_{l-i}}(t)}\dot{\xi}_{i}^{\eps_1,\dots,\eps_{l-i}}(t)-
Y(t,\dot{\xi}_{i}^{\eps_1,\dots,\eps_{l-i}}(t)) \nonumber
\\ =\sum_{a=1}^{N_i}\lambda_{a}^{\eps_1,\dots,\eps_{l-i}}(t) Z_a^i
({\xi}_{i}^{\eps_1,\dots,\eps_{l-i}}(t)) \label{e-para}
\end{align}
with $\lambda_{a}^{\eps_1,\dots,\eps_{l-i}}\in
\mathcal{C}(I,[0,+\infty))$ and $Z_a^i
\in \HK_i$, then $\xi_{i-1}^{\eps_1,\dots,\eps_{l-i+1}}$ is defined
as the solution to
\begin{align}&
\nabla_{\dot{\xi}_{i-1}^{\eps_1,\dots,\eps_{l-i+1}}(t)}\dot{\xi}_{i-1}^{\eps_1,\dots,\eps_{l-i+1}}(t) \nonumber\\
& -Y(t,\dot{\xi}_{i-1}^{\eps_1,\dots,\eps_{l-i+1}}(t)) \nonumber \\
&=
\sum_{a=1}^{N_i}\Big(\lambda_{a}^{\eps_1,\dots,\eps_{l-i}}(t)F_a^{i-1}
({\xi}_{i-1}^{\eps_1,\dots,\eps_{l-i+1}}(t))+\nonumber \\ &
\frac{1}{\eps_{l-i+1}}\sum_{b=1}^{\hat
N_a^i}\psi_{\lo(a,b)}\left(\frac{t}{\eps_{l-i+1}}\right)\sqrt{
\lambda_{a}^{\eps_1,\dots,\eps_{l-i}}(t)\alpha^{i-1}_{a,b}}\nonumber
\\ &
G_{a,b}^{i-1}
({\xi}_{i-1}^{\eps_1,\dots,\eps_{l-i+1}}(t))\Big)\label{extras}
\end{align}
with
$\dot{\xi}_{i-1}^{\eps_1,\dots,\eps_{l-i+1}}(0)=\dot{\gamma}_{\mathrm{ref}}(0)$,
where
$$Z_a^i
=F_a^{i-1}
-\sum_{b=1}^{\hat N^i_a}\alpha^{i-1}_{a,b}
\langle G_{a,b}^{i-1}
\colon G_{a,b}^{i-1}
\rangle
$$
and $F_a^{i-1}
\in \HK_{i-1}$, $
G_{a,b}^{i-1}
\in {\rm L}(\HK_{i-1})$, $\alpha^{i-1}_{a,b}\geq 0$.
Recall that $(\psi_j)_{j\in\N}$ is a $\Lambda_T$-orthonormal and zero-mean sequence, for some $T>0$.
Each
$\lambda_a^{\eps_1,\dots,\eps_{l-i+1}}(\cdot)$ is either equal to
some $\lambda_{\tilde a}^{\eps_1,\dots,\eps_{l-i}}(\cdot)$ or is of
the form
$$\frac{\sqrt{
\lambda_{\tilde
a}^{\eps_1,\dots,\eps_{l-i}}(\cdot)\alpha^{i-1}_{\tilde
a,b}}}{\eps_{l-i+1}}\max\left\{\upsilon \psi_{\lo(\tilde
a,b)}\left(\frac{\cdot}{\eps_{l-i+1}}\right),0\right\}$$ with
$\upsilon$ equal to $1$ or $-1$.

We choose $(\psi_j)_{j\in\N}$ as follows: we require $\psi_1$ to be positive in $(0,T/2)$ and to annihilate, together with all its derivatives, at $0$ and $T/2$. We also require it to satisfy
$$\psi_1\left(\frac T2-t\right)=\psi_1(t),\quad t\in\left[0,\frac T2\right],$$
and we extend it by
$$\psi_1(t)=-\psi_1\left(t-\frac T2\right),\quad t\in\left[\frac T2,T\right],$$
and by $T$-periodicity over $\R$. Finally we normalize $\psi_1$ in such a way that
$\Lambda_T(\psi_1,\psi_1)=1$.
Then we define $\psi_j$ by
$$\psi_j(t)=2^j \psi_1( 2^j t).$$

Such choice of $(\psi_j)_{j\in\N}$  is motivated by the  property
that, for every choice of $j\in\N$, $l\in \N_0$ and $\upsilon\in\{-1,1\}$,
the function
$
t\mapsto \sqrt[2^l]{\max\left\{\upsilon \psi_{j}\left(t \right),0\right\}}
$
is smooth.
In particular, by backward recursion, each $\lambda_a^{\eps_1,\dots,\eps_{i}}$ is smooth and is
the product of
functions of the type $\sqrt[2^l]{\frac{1}{\eps_{h}}\max\left\{\upsilon \psi_{j}\left(\frac{\cdot}{\eps_{h}} \right),0\right\}}$
and of
$\sqrt[2^m]{\lambda_b}$ for some $m\in \N_0$ and some $b\in\{1,\dots,N_l\}$.

An important consequence of this factorization, which will be exploited in the proof of Theorem~\ref{thm-1-parameter}, is that
 the 
 derivatives of $\sqrt{\lambda_a^{\eps_1,\dots,\eps_{i}}}$ with respect to time can be bounded by a finite constant depending explicitly on $\eps_1,\dots,\eps_{l-1}$.

The smoothness of $\lambda_a^{\eps_1,\dots,\eps_{i}}$, moreover,
allows us to consider $\lambda_a^{\eps_1,\dots,\eps_{i}} Z_a^{l-i}$
as a smooth vector field on the extended manifold $\R\times Q$ and,
similarly, $\lambda_a^{\eps_1,\dots,\eps_{l-i}} (Z_a^i)^V$ as a
smooth vector field on the extended manifold $\R\times TQ$.

Summing up,  the trajectories of the
$l$-parameter family
$\xi_0^{\eps_1,\dots,\eps_l}=\xi^{\eps_1,\dots,\eps_l}$ are driven
by a $l$-parameter family of control laws
$u^{\eps_1,\dots,\eps_l}\in \mathcal{C}^\infty(I,\mathbb{R}^k)$ depending
smoothly on $\eps_1,\dots,\eps_{l}$.
 The
construction of Theorem~\ref{GeneralTrack} can be summarized as
follows: given $\eps>0$, if
\begin{equation}\label{to-be-quantified}
0<\eps_l\ll \eps_{l-1}\ll\cdots \ll \eps_1\ll 1
\end{equation}
then ${\rm d}(\xi^{\eps_1,\dots,\eps_l}(t),\gref(t))<\eps$ for every
$t\in I$. Our aim is here to quantify the relations in
(\ref{to-be-quantified}).
More precisely, we introduce $l-1$ functions
$\eta_2,\dots,\eta_l:(0,+\infty)\rightarrow(0,+\infty)$ and we look
for asymptotic conditions on their convergence to zero at zero such
that
\begin{equation}\label{oneP}
\lim_{\eps\rightarrow0}{\rm d}(\xi^{\eps,\eta_2(\eps),\eta_3\circ
\eta_2(\eps),\dots,\eta_l\circ \cdots \circ
\eta_{2}(\eps)}(t),\gref(t))=0
\end{equation}
uniformly with respect to $t\in I$. Let $\hat \eta_i=
\eta_{i}\circ\cdots\circ\eta_2$ for $i=2,\dots,l$ and define
$\hat\eta_1$ as the identity on $(0,+\infty)$. We say that
$\eps\mapsto u^{\hat\eta_1(\eps),\dots,\hat\eta_l(\eps)}$ is a {\it
one\textendash{}parameter tracking control law for $\gref$} if (\ref{oneP})
holds true.

\begin{thm}\label{thm-1-parameter}
Let $\Sigma=(Q,\nabla,Y,\mathscr{Y},\mathbb{R}^k)$ be a FACCS. Let
$\HK_i$, $i\in \mathbb{N}_0$, be defined as in (\ref{setH}). Fix a
reference trajectory $\gref\in \mathcal{C}^\infty(I,Q)$ and assume
that there exists $l\in\mathbb{N}$ such that
\[\nabla_{\dot{\gamma}_{\mathrm{ref}}(t)}\dot{\gamma}_{\mathrm{ref}}(t)-Y(t,\dot{\gamma}_{\mathrm{ref}}(t))\in \HK_l(\gref(t))\]
for every $t\in I$. Construct $\xi^{\eps_1,\dots,\eps_l}$,
$u^{\eps_1,\dots,\eps_l}$ and $\hat\eta_i$ as above. If
$\eta_i:(0,+\infty)\rightarrow(0,+\infty)$ satisfies
$\limsup_{\eps\rightarrow0}\eta_i(\eps)/\eps^3<\infty$ for every
$i=2,\dots,l$, then $\eps\mapsto
u^{\hat\eta_1(\eps),\dots,\hat\eta_l(\eps)}$ is a
one\textendash{}parameter tracking control law for $\gref$.
\end{thm}
\begin{proof}
The first step of the proof consists in estimating the order
with respect to $\eps$ of the $L^\infty$-norm of the time-dependent
parameters appearing in  the parameterization (\ref{e-para}).
We write $\eps_i$ for
$\hat\eta_i(\eps)$.
Denoting by $C$ any constant not depending on the $\eps_j$,
it is
easy to check by backward induction on $i=0,\dots,l$ that
$$\|\lambda_a^{\eps_1,\dots,\eps_i}\|_{\infty}\leq  \frac C{\sqrt[2^{i-1}]{\eps_1}\sqrt[2^{i-2}]{\eps_2}\cdots \sqrt{\eps_{i-1}}{\eps_i}}.$$
Exploiting the factorization of $\lambda_a^{\eps_1,\dots,\eps_i}$
described above we get, in addition,
$$
 \left\|\frac{{\rm d}^m}{{\rm d}t^m} \sqrt{\lambda_a^{\eps_1,\dots,\eps_i}}\right\|_{\infty}\leq  \frac C{\sqrt[2^{i}]{\eps_1}\sqrt[2^{i-1}]{\eps_2}\cdots \sqrt[4]{\eps_{i-1}}{\eps_i}^{m+\frac12}}
$$
for 
$m\in\mathbb{N}_0$.

Consider the extended system on $\mathbb{R}\times TQ$ associated
with (\ref{first}) where the time is the new variable. In the
following computations we write, using the notation introduced in
(\ref{extras}),
\begin{eqnarray*}
X^e&=&(1,Z+Y^V),\\
\Phi_a^i&=&(0,(F_a^i)^V),\\
\Gamma_{a,b}^i&=&  (0,(G_{a,b}^i)^V),
\end{eqnarray*}
all seen as vector fields on $\mathbb{R}\times TQ$.
We also
define ${\gamma}_{\mathrm{ref}}^e(0)=(0,{\dot\gamma}_{\mathrm{ref}}(0))$
and, given a smooth function $\lambda:I\to \R$, we write  $\lambda^V$ to denote a smooth function on $\mathbb{R}\times TQ$ such that $\lambda^V(t,v)=\lambda(t)$ for every $t\in I$ and every $v\in TQ$.
In particular, we define
\begin{eqnarray*}
\theta_{a,b}^{i}&=&\left(
{\sqrt{
\lambda_{a}^{\eps_1,\dots,\eps_{i}}\alpha^{l-i-1}_{a,b}}}\right)
^V,
\\
\dot{\theta}_{a,b}^{i}&=&\left(\frac{{\rm d}}{{\rm d}t}
{\sqrt{
\lambda_{a}^{\eps_1,\dots,\eps_{i}}\alpha^{l-i-1}_{a,b}}}\right)
^V.
\end{eqnarray*}

Then, applying (\ref{nilpotent_exponential})
 and (\ref{variation_formula}),
\[\begin{array}{l} (t,\dot{\xi}_0^{\eps_1,\ldots,\eps_l}(t))=\\
\overrightarrow{{\rm exp}}\int^t_0
(X^e +\sum_{a=1}^{N_0}\lambda_a^{\eps_1,\dots,\eps_l}(s)
(0,(Z_a^0)^V) ){\rm
d}s\;{\gamma}_{\mathrm{ref}}^e(0)\\ =
\overrightarrow{{\rm exp}}\int^t_0
\Big(X^e + \sum_{a=1}^{N_1}\lambda_a^{\eps_1,\dots,\eps_{l-1}}(s)
\Phi_a^0
\\+\frac{1}{\eps_l}
\sum_{a=1}^{N_1}
\sum_{b=1}^{\hat
N_a^1}\psi_{\lo(a,b)}(\frac{s}{\eps_{l}})\sqrt{
\lambda_{a}^{\eps_1,\dots,\eps_{l-1}}(s)\alpha^{0}_{a,b}}
\\
\Gamma_{a,b}^{0}\Big){\rm
d}s\;{\gamma}_{\mathrm{ref}}^e(0)\end{array}\]
\[\begin{array}{l}
=\overrightarrow{{\rm exp}}\int^t_0\frac{1}{\eps_l}\sum_{a=1}^{N_1}
\sum_{b=1}^{\hat N_a^1}\psi_{\lo(a,b)}(\frac{s}{\eps_{l}})
\theta_{a,b}^{l-1}\Gamma_{a,b}^{0}{\rm d}s
\,\circ\\
\overrightarrow{{\rm exp}}\int^t_0 \Big(X^e+\sum_{a=1}^{N_1}\lambda_a^{\eps_1,\dots,\eps_{l-1}}(s)
\Phi_a^0\\
+
\sum_{a=1}^{N_1}
\sum_{b=1}^{\hat
N_a^1}
\int^s_0 \frac{1}{\eps_l}
\psi_{\lo(a,b)}(\frac{s_0}{\eps_{l}})
{\rm
d}s_0(
\theta_{a,b}^{l-1}
[\Gamma_{a,b}^{0},X^e]\\
- \dot{\theta}_{a,b}^{l-1}
\Gamma_{a,b}^{0})
- \sum_{a,a'=1}^{N_1}
\sum_{b=1}^{\hat
N_a^1}\sum_{b'=1}^{\hat
N_{a'}^1}\int^s_0\frac{1}{\eps_l}
\psi_{\lo(a,b)}(\frac{s_1}{\eps_{l}})
\\ \int^{s_1}_0\frac{1}{\eps_l}\psi_{\lo(a',b')}(\frac{s_0}{\eps_{l}})
{\rm d}s_0{\rm
d}s_1
\theta_{a,b}^{l-1}\theta_{a',b'}^{l-1}\\
(0,\langle G_{a,b}^{0} \colon G_{a',b'}^{0} \rangle^V)\Big) {\rm d}s
 {\gamma}_{\mathrm{ref}}^e(0).
 \end{array}\]
 From now
on, let us
denote by
$\mathscr{V}$ any {\it vertical flow}, i.e., any flow on $\R\times TQ$
preserving the base point on $Q$.

 Notice that, by construction,
 \[\begin{array}{l}
 \sum_{a=1}^{N_1}\lambda_a^{\eps_1,\dots,\eps_{l-1}}(s)
\Phi_a^0+\\
-\sum_{a,a'=1}^{N_1}
\sum_{b=1}^{\hat
N_a^1}\sum_{b'=1}^{\hat
N_{a'}^1}\Lambda_T(\psi_{\lo(a,b)},\psi_{\lo(a',b')})\\[1mm]
 \sqrt{
\lambda_{a}^{\eps_1,\dots,\eps_{l-1}}(s)
\lambda_{a}^{\eps_1,\dots,\eps_{l-1}}(s)\alpha^{0}_{a,b}\alpha^{0}_{a',b'}}\\(0,\langle G_{a,b}^{0} \colon G_{a',b'}^{0} \rangle^V)=
\sum_{a=1}^{N_1}\lambda_a^{\eps_1,\dots,\eps_l}(s)
(0,(Z_a^1)^V).
  \end{array}\]
Then,
\[\begin{array}{l} (t,\dot{\xi}_0^{\eps_1,\ldots,\eps_l}(t))=\\  {\mathscr{V}}\circ \;
\overrightarrow{{\rm exp}}\int^t_0
\Big(X^e+\sum_{a=1}^{N_1}\lambda_a^{\eps_1,\dots,\eps_{l-1}}(s)
(0,(Z_a^1)^V)\\
+
\sum_{a=1}^{N_1}
\sum_{b=1}^{\hat
N_a^1}
\int^s_0 \frac{1}{\eps_l}
\psi_{\lo(a,b)}(\frac{s_0}{\eps_{l}})
{\rm
d}s_0(
\theta_{a,b}^{l-1}
[\Gamma_{a,b}^{0},X^e]\\
- \dot{\theta}_{a,b}^{l-1} \Gamma_{a,b}^{0}) \\ +
\sum_{a,a'=1}^{N_1} \sum_{b=1}^{\hat N_a^1}\sum_{b'=1}^{\hat
N_{a'}^1} (\Lambda_T(\psi_{\lo(a,b)},\psi_{\lo(a',b')})\\-
\int^s_0\frac{1}{\eps_l} \psi_{\lo(a,b)}(\frac{s_1}{\eps_{l}})
 \int^{s_1}_0\frac{1}{\eps_l}\psi_{\lo(a',b')}(\frac{s_0}{\eps_{l}})
{\rm d}s_0{\rm
d}s_1)\\(0,\langle G_{a,b}^{0} \colon G_{a',b'}^{0} \rangle^V) \theta_{a,b}^{l-1}\theta_{a',b'}^{l-1}\Big){\rm
d}s
{\gamma}_{\mathrm{ref}}^e(0).
\end{array}\]

Applying iteratively the same computation as above, one ends up with
\[\begin{array}{l} (t,\dot{\xi}_0^{\eps_1,\ldots,\eps_l}(t))=\\
 \mathscr{V}\,\circ  \overrightarrow{{\rm exp}}\int^t_0
\left(X^e+\lambda_a(s) (0,(Z_a^l)^V)+  \mathscr{T}(s)\right){\rm
d}s\;{\gamma}^e_{\mathrm{ref}}(0),
\end{array}\]
where $\mathscr{T}(s)$ is a sum of terms of the form $\zeta(s)V$ where
$V$ is a vector field on $\R\times TQ$ independent of the $\eps_j$,
while $\zeta:I\to \R$ is smooth, depends on the $\eps_j$ and is of one
of the following four types:
\[\begin{array}{l}
\zeta_1(s)=
\int^s_0 \frac{1}{\eps_i}
\psi_{\lo(a,b)}(\frac{s_0}{\eps_{i}})
{\rm
d}s_0\,
\sqrt{\lambda_a^{\eps_1,\dots,\eps_{i-1}}(s)\alpha_{a,b}^{l-i}},\\ 
\zeta_2(s)= -\int^s_0 \frac{1}{\eps_i}
\psi_{\lo(a,b)}(\frac{s_0}{\eps_{i}}) {\rm d}s_0 
\frac{{\rm
d}}{{\rm
d}s}\sqrt{\lambda_a^{\eps_1,\dots,\eps_{i-1}}(s)\alpha_{a,b}^{l-i}}
,\\ 
\zeta_3(s)=\frac 1{1+\delta_{bc}}\Big[2 (\Lambda_T(\psi_{\lo(a,b)},\psi_{\lo(a',b')}))\\-
\left(\int^s_0\frac{1}{\eps_i}
\psi_{\lo(a,b)}(\frac{s_0}{\eps_{i}}){\rm d}s_0\right) \left(
 \int^{s}_0\frac{1}{\eps_i}\psi_{\lo(a',b')}(\frac{s_0}{\eps_{i}})
{\rm d}s_0\right)\Big]\\
\sqrt{\lambda_a^{\eps_1,\dots,\eps_{i-1}}(s)\lambda_{a'}^{\eps_1,\dots,\eps_{i-1}}(s)\alpha_{a,b}^{l-i}\alpha_{a',b'}^{l-i}}
\end{array}\]
for
$i=1,\ldots, l$, or
\[\begin{array}{l}
\zeta_4(s)= -\left(\int^s_0\frac{1}{\eps_j}
\psi_{\lo(a,b)}(\frac{s_0}{\eps_{j}}){\rm d}s_0\right)\\
\left(\int^s_0\frac{1}{\eps_i}
\psi_{\lo(a',b')}(\frac{s_0}{\eps_{i}}){\rm d}s_0\right)\\
\sqrt{\lambda_a^{\eps_1,\dots,\eps_{j-1}}(s)\lambda_{a'}^{\eps_1,\dots,\eps_{i-1}}(s)\alpha_{a,b}^{l-j}\alpha_{a',b'}^{l-i}}
 \end{array}\]
 for $i>j$, $i,j\in\{1,\ldots, l\}$.

According to Lemma~\ref{ConvIntegral}, the theorem is proved if we show that, for every $\zeta$ of one of the four types introduced above,
$\int_0^t\zeta(s){\rm d}s$ converges to zero uniformly with respect to
$t\in I$ as $\eps$ goes to zero.

This can be done by applying
Lemma~\ref{simple_fact}.
 Taking for instance
 $$f(\tau,s)=
 \int^\tau_0 \psi_{\lo(a,b)}(s_0)
{\rm
d}s_0\,
\sqrt{\lambda_a^{\eps_1,\dots,\eps_{i-1}}(s)\alpha_{a,b}^{l-i}},
 $$
and $\hat\eps=\eps_i$
 leads to
 $$\int_0^t \zeta_1(s) {\rm
d}s\leq C
\frac{\eps_i }{\sqrt[2^{i-1}]{\eps_1}\sqrt[2^{i-2}]{\eps_2}\cdots \sqrt[4]{\eps_{i-2}}\eps_{i-1}^{\frac32}}.$$
Similarly,
$$\int_0^t \zeta_2(s) {\rm
d}s \leq C  \frac{\eps_i }{\sqrt[2^{i-1}]{\eps_1}\sqrt[2^{i-2}]{\eps_2}\cdots \sqrt[4]{\eps_{i-2}}\eps_{i-1}^{\frac52}}.$$
Taking
\[\begin{array}{rl}
f(\tau,s)=&\left(\int^\tau_0 \psi_{\lo(a,b)}(s_0){\rm d}s_0\right)
\left(
 \int^{\tau}_0\psi_{\lo(a',b')}(s_0)
{\rm d}s_0\right)\\[1mm]
&\sqrt{\lambda_a^{\eps_1,\dots,\eps_{i-1}}(s)\lambda_{a'}^{\eps_1,\dots,\eps_{i-1}}(s)\alpha_{a,b}^{l-i}\alpha_{a',b'}^{l-i}}
\end{array}\]
we obtain
 $$\int_0^t \zeta_3(s) {\rm
d}s\leq C
\frac{\eps_i }{\sqrt[2^{i-2}]{\eps_1}\sqrt[2^{i-3}]{\eps_2}\cdots \sqrt{\eps_{i-2}}\eps_{i-1}^{2}}.$$
Finally, with
\[\begin{array}{l}
f(\tau,s)=
\left(\int^{s/\eps_j}_0
\psi_{\lo(a,b)}(s_0){\rm d}s_0\right)
\left(\int^\tau_0
\psi_{\lo(a',b')}(s_0){\rm d}s_0\right)\\
\ \ \ \ \ \ \ \ \ \ \ \ \sqrt{\lambda_a^{\eps_1,\dots,\eps_{j-1}}(s)\lambda_{a'}^{\eps_1,\dots,\eps_{i-1}}(s)\alpha_{a,b}^{l-j}\alpha_{a',b'}^{l-i}}
\end{array}\]
we get
\[\begin{array}{l}
\int_0^t \zeta_4(s) {\rm
d}s\leq C
\frac{\eps_i }{
\sqrt[2^{j-1}]{\eps_1}\sqrt[2^{j-2}]{\eps_2}\cdots \sqrt[4]{\eps_{j-1}}\sqrt{\eps_{j-1}}}\\
\ \ \ \ \ \ \ \ \ \ \ \  \ \ \ \ \ \ \ \ \frac{1}{
\sqrt[2^{i-1}]{\eps_1}\sqrt[2^{i-2}]{\eps_2}\cdots \sqrt[4]{\eps_{i-2}}\eps_{i-1}^{\frac 32}}.
\end{array}\]
Notice that the upper bound for $\int_0^t \zeta_2(s) {\rm
d}s$ is the one growing faster as $\eps$ goes to zero. Indeed, in order to compare
it with the one for $\int_0^t \zeta_3(s) {\rm
d}s$ it suffices to notice that
\[\begin{array}{l}
\frac{\sqrt[2^{i-1}]{\eps_1}\sqrt[2^{i-2}]{\eps_2}\cdots \sqrt[4]{\eps_{i-2}}\eps_{i-1}^{\frac52}}{\sqrt[2^{i-2}]{\eps_1}\sqrt[2^{i-3}]{\eps_2}\cdots \sqrt{\eps_{i-2}}\eps_{i-1}^{2}}=\\[5mm]
\sqrt{\frac{\eps_{i-1}}{\eps_{i-2}}}\sqrt[4]{\frac{\eps_{i-2}}{\eps_{i-3}}}\cdots\sqrt[2^{i-2}]{\frac{\eps_{2}}{\eps_{1}}}\sqrt[2^{i-1}]{\eps_1}
\end{array}\]
converges to zero as $\eps$ goes to zero.

Hence, for each
$\zeta$ as above, there exists $i=1,\dots,l$ such that
$$
 \int_0^t\zeta(s) {\rm
d}s\leq C\frac{\eps_i }{\sqrt[2^{i-1}]{\eps_1}\sqrt[2^{i-2}]{\eps_2}\cdots \sqrt[4]{\eps_{i-2}}\eps_{i-1}^{\frac 52}}
$$
and we are left to notice that
\[\begin{array}{l}
\frac{\eps_i }{\sqrt[2^{i-1}]{\eps_1}\sqrt[2^{i-2}]{\eps_2}\cdots \sqrt[4]{\eps_{i-2}}\eps_{i-1}^{\frac 52}}=\\[5mm]
\frac{\eps_i}{\eps_{i-1}^3}\sqrt{\frac{\eps_{i-1}}{\eps_{i-2}}}\sqrt[4]{\frac{\eps_{i-2}}{\eps_{i-3}}}\cdots\sqrt[2^{i-2}]{\frac{\eps_{2}}{\eps_{1}}}\sqrt[2^{i-1}]{\eps_1}
\end{array}\]
 tends to zero as $\eps$ goes to zero. \end{proof}

\begin{remark}\label{H_quantification}
The hypothesis
$\limsup_{\eps\rightarrow0}\eta_i(\eps)/\eps^3<\infty$
has been chosen, in the statement of Theorem~\ref{thm-1-parameter},
because of its simplicity.
However, we can weaken it by requiring that
$\limsup_{\eps\rightarrow0}\eta_i(\eps)/\eps^{\frac52+a}<\infty$
with
$a=\frac{\sqrt{5}}{2}-1\simeq 0.12$. Indeed, with this choice of
$a$,
\[\begin{array}{l}
\frac{\eps_j^a }{\sqrt[2^{j}]{\eps_1}\sqrt[2^{j-1}]{\eps_2}\cdots \sqrt[4]{\eps_{j-1}}}
=\left(
\frac{\eps_j}{\eps_{j-1}^{\frac52+a}}
\right)^a\\
\left(
\frac{\eps_{j-1}^a}{\sqrt[2^{j-1}]{\eps_1}\sqrt[2^{j-2}]{\eps_2}\cdots \sqrt[4]{\eps_{j-2}}}
\right)^{\frac12},
\end{array}\]
so that
\[\begin{array}{l}
\frac{\eps_i }{\sqrt[2^{i-1}]{\eps_1}\sqrt[2^{i-2}]{\eps_2}\cdots \sqrt[4]{\eps_{i-2}}\eps_{i-1}^{\frac 52}
}=\\
\frac{\eps_i}{\eps_{i-1}^{\frac52+a}}\left(
\frac{\eps_{i-1}}{\eps_{i-2}^{\frac52+a}}
\right)^a
\left(
\frac{\eps_{i-2}}{\eps_{i-3}^{\frac52+a}}
\right)^\frac{a}2\cdots \left(
\frac{\eps_2}{\eps_{1}^{\frac52+a}}
\right)^\frac{a}{2^{i-3}}
\eps_1^\frac{a}{2^{i-2}}.
\end{array}\]
Hence, each $ \int_0^t\zeta(s) {\rm d}s$ goes to zero uniformly with
respect to $t\in I$ as $\eps$ tends to zero.
\end{remark}

\subsection{Case $\mathscr{Z}$}\label{SCaseZ}

Analogously to Section \ref{SCaseH}, the aim is to provide an
algorithmic implementation of the results obtained in Section
\ref{SgeneralTrack} about the existence of controls yielding
tracking but this time under the hypotheses of
Corollary~\ref{CorolSym1}. Instead of adopting the algorithmic
scheme on which the proof of Theorem~\ref{GeneralTrack} is based, as
done in Section \ref{SCaseH}, we rely here on the iteration of
the scheme proposed in \cite{2005BulloAndrewBook} and recalled in
Theorem~\ref{ACCStrack} (see also Remark~\ref{betterT}). The advantage is that, under the more restrictive hypotheses of Corollary~\ref{CorolSym1}, we can base the iteration scheme
on the $\Lambda_T$-orthonormal and zero-mean sequence defined in (\ref{varphi}) using trigonometric functions, which is more convenient for numerical implementation than the sequence $(\psi_j)_{j\in\N}$ constructed in the previous section.

We start from a parameterization
\begin{equation}\label{Eq-First-Param-Case Z}\nabla_{\dot{\gamma}_{\mathrm{ref}}(t)}\dot{\gamma}_{\mathrm{ref}}(t)-Y(t,\dot{\gamma}_{\mathrm{ref}}(t))
=\sum_{a=1}^{N_l}\lambda_a(t) Z^l_a(\gref(t))\end{equation} of
$\gref\colon I \rightarrow Q$, with $Z_a^l\in \mathscr{Z}_l$  and $\lambda_a$ smooth on $I$ for every
$a=1,\dots,N_l$. (For the definition of $\mathscr{Z}_l$, see (\ref{setZ}).)
As in Section \ref{SCaseH}, we can construct
algorithmically a $l$-parameter family of admissible trajectories
$\xi^{\eps_1,\dots,\eps_l}$ of $\Sigma$ with $\eps_1,\dots,\eps_l>0$ such that
$\xi_l=\gref$,
the uniform
limit in (\ref{cascade-limits}) %
exists for every $i=1,\dots,l$ and every
$\eps_1,\dots,\eps_{l-i}>0$, and, instead of
(\ref{EqParamCaseH}), it satisfies
\begin{eqnarray*}\lefteqn{\nabla_{\dot{\xi}_{i}^{\eps_1,\dots,\eps_{l-i}}(t)}\dot{\xi}_{i}^{\eps_1,\dots,\eps_{l-i}}(t)
-Y(t,\dot{\xi}_{i}^{\eps_1,\dots,\eps_{l-i}}(t))}\\ && \in
\mathrm{span}_{\mathbb{R}}(\mathscr{Z}_{i}({\xi}_{i}^{\eps_1,\dots,\eps_{l-i}}(t))).\end{eqnarray*}

The algorithm works by applying, at each step, the construction of Theorem~\ref{ACCStrack}
with $\mathscr{Z}_i$ as $\mathscr{Y}$ (see Remark~\ref{betterT}). Then by backward recursion on
$i$, the solutions to 
\begin{equation*}\label{Eq-Stepi}\begin{array}{l}\nabla_{\dot{\xi}_{i}^{\eps_1,\dots,\eps_{l-i}}(t)}\dot{\xi}_{i}^{\eps_1,\dots,\eps_{l-i}}(t)
-Y(t,\dot{\xi}_{i}^{\eps_1,\dots,\eps_{l-i}}(t))\\=\displaystyle\sum_{a=1}^{N_i}
\lambda_a^{\eps_1,\dots,\eps_{l-i}}(t)Z_a^{i-1}({\xi}_{i}^{\eps_1,\dots,\eps_{l-i}}(t))
\\+ \displaystyle\sum_{b,c=1,b<c}^{N_i}
\lambda_{b,c}^{\eps_1,\dots,\eps_{l-i}}(t)\langle Z_b^{i-1}\colon
Z_c^{i-1}\rangle
({\xi}_{i}^{\eps_1,\dots,\eps_{l-i}}(t))\end{array}\end{equation*}
for $Z_a^{i-1},Z_b^{i-1},Z_c^{i-1}\in \mathscr{Z}_{i-1}$ are
trackable by solutions to
\begin{equation}\label{Eq-Stepi-1}\begin{array}{l}\nabla_{\dot{\xi}_{i-1}^{\eps_1,\dots,\eps_{l-i+1}}(t)}\dot{\xi}_{i-1}^{\eps_1,\dots,\eps_{l-i+1}}(t)
-Y(t,\dot{\xi}_{i}^{\eps_1,\dots,\eps_{l-i+1}}(t))\\=\displaystyle\sum_{a=1}^{N_{i-1}}\left(u_{{\rm
slow},a}^{\eps_1,\dots,\eps_{l-i}}(t
) \right.\\ \left.+ \frac{1}{\eps_{l-i+1}}u_{{\rm
osc},a}^{\eps_1,\dots,\eps_{l-i}}(t/\eps_{l-i+1},t)
\right)Z_a^{i-1}({\xi}_{i-1}^{\eps_1,\dots,\eps_{l-i+1}}(t))\end{array}\end{equation}
where $u_{{\rm slow},a}^{\eps_1,\dots,\eps_{l-i}}$ and $u_{{\rm
osc},a}^{\eps_1,\dots,\eps_{l-i}}$ are constructed as in
Theorem~\ref{ACCStrack}. Notice that
$$
u_{{\rm
slow},a}^{\eps_1,\dots,\eps_{l-i}}(t
) + \frac{1}{\eps_{l-i+1}}u_{{\rm
osc},a}^{\eps_1,\dots,\eps_{l-i}}(t/\eps_{l-i+1},t)
$$
plays the role of a $\lambda_{\hat
a}^{\eps_1,\dots,\eps_{l-i+1}}(t)$ or a $\lambda_{\hat b,\hat
c}^{\eps_1,\dots,\eps_{l-i+1}}(t)$ at the next step.

We
have \begin{small}$$\|u_{{\rm
slow},a}^{\eps_1,\dots,\eps_{l-i}}\|_\infty \leq C
\max_{b,c}(\|\lambda_b^{\eps_1,\dots,\eps_{l-i}}\|_\infty,\|\lambda_{b,c}^{\eps_1,\dots,\eps_{l-i}}\|_\infty^2),$$
$$ \|u_{{\rm
osc},a}^{\eps_1,\dots,\eps_{l-i}}\|_\infty \leq
C\max_{b,c}(\|\lambda_{b,c}^{\eps_1,\dots,\eps_{l-i}}\|_\infty).
 $$\end{small}

Given $\alpha, \beta\colon [0,T]\times I\to \mathbb{R}$, define
$\Lambda_T$ as
\[\begin{array}{l}\Lambda_T(t,\alpha,\beta)= \\ \quad \frac{1}{2T}\int_0^{T}\left(\int_0^\tau \alpha(s,t){\rm
d}s\right)\left(\int_0^\tau \beta(s,t) {\rm d}s\right){\rm d}\tau.
\end{array}\] 
The construction of $u_{{\rm slow}}$ and $u_{{\rm osc}}$ is such that
\begin{equation*}\label{Eq-Lambda} \begin{array}{l}\displaystyle\sum_{a=1}^{N_{i-1}}
u_{{\rm slow},a}^{\eps_1,\dots,\eps_{l-i}}(t)Z_a^{i-1}\\
-\sum_{b,c=1}^{N_{i-1}}
\Lambda_T(t,u_{{\rm osc},b}^{\eps_1,\dots,\eps_{l-i}},u_{{\rm osc},c}^{\eps_1,\dots,\eps_{l-i}})\langle Z_b^{i-1}\colon Z_c^{i-1}\rangle\\
=\displaystyle\sum_{a=1}^{N_{i-1}}
\lambda_a^{\eps_1,\dots,\eps_{l-i}}(t)Z_a^{i-1}\\
+\displaystyle\sum_{b,c=1,b<c}^{N_{i-1}}
\lambda_{b,c}^{\eps_1,\dots,\eps_{l-i}}(t)\langle Z_b^{i-1}\colon
Z_c^{i-1}\rangle  \end{array}\end{equation*}
(see \cite{2005BulloAndrewBook}).
As in Section \ref{SCaseH}, our aim is here to quantify the
relations in (\ref{to-be-quantified}). We will use the same
notations for $\eta_i$ and $\hat{\eta}_i$ as in
Section~\ref{SCaseH}. The notion of being a
one\textendash{}parameter tracking control law for $\gref$ is again defined
through (\ref{oneP}).

\begin{thm}\label{thm-1-parameter-Z}
Let $\Sigma=(Q,\nabla,Y,\mathscr{Y},\mathbb{R}^k)$ be a FACCS. Let
$\mathscr{Z}_i$, $i\in \mathbb{N}_0$, be defined as in (\ref{setZ})
satisfying the same hypotheses as in Corollary \ref{CorolSym1}. 
 Fix
a reference trajectory $\gref\in \mathcal{C}^\infty(I,Q)$ and assume
that there exists $l\in\mathbb{N}$ such that
\[\nabla_{\dot{\gamma}_{\mathrm{ref}}(t)}\dot{\gamma}_{\mathrm{ref}}(t)-Y(t,\dot{\gamma}_{\mathrm{ref}}(t))\in \mathscr{Z}_l(\gref(t))\]
for every $t\in I$. Construct $\xi^{\eps_1,\dots,\eps_l}$,
$u^{\eps_1,\dots,\eps_l}$ and $\hat\eta_i$ as above. If
$\eta_i:(0,+\infty)\rightarrow(0,+\infty)$ satisfies
$\limsup_{\eps\rightarrow 0}\eta_i(\eps)/\eps^4<\infty$ for every
$i=2,\dots,l$, then $\eps\mapsto
u^{\hat\eta_1(\eps),\dots,\hat\eta_l(\eps)}$ is a
one\textendash{}parameter tracking control law for $\gref$.
\end{thm}
\begin{proof} The first step of the proof consists in estimating the order
with respect to $\eps$ of the ${\rm L}^{\infty}$\textendash{}norm of
the time\textendash{}dependent parameters appearing in
the parameterization (\ref{Eq-Stepi-1}). Once more we write $\eps_i$
for $\hat\eta_i(\eps)$. By induction it can be proved
that
\begin{eqnarray*}
 \|\partial_2^m u_{{\rm osc},a}^{\eps_1,\ldots,\eps_{i}}\|_{\infty}&\leq& C
(\eps_1\cdots\eps_{i-1}\eps_i^{m+1})^{-1}
\end{eqnarray*}
for every $m\in \mathbb{N}_0$.

Consider, as in the proof of Theorem~\ref{thm-1-parameter}, the extended system on $\mathbb{R}\times TQ$ associated with
(\ref{first}). In the following
computations we write
\begin{eqnarray*}
V^i_a&=&(0,(Z_a^i)^V), \\
W^i_{b,c}&=&(0,\langle Z^i_b\colon Z^i_c\rangle^V).
\end{eqnarray*}
We also write
$u_{{\rm
osc},a}^{\eps_1,\dots,\eps_{i}}(\sigma)$ to denote
$(u_{{\rm
osc},a}^{\eps_1,\dots,\eps_{i}}(\sigma,\cdot))^V$, i.e., $u_{{\rm
osc},a}^{\eps_1,\dots,\eps_{i}}(\sigma)$ is the smooth function on
$\R\times TQ$ such that
$$u_{{\rm
osc},a}^{\eps_1,\dots,\eps_{i}}(\sigma)(t,v)=
u_{{\rm
osc},a}^{\eps_1,\dots,\eps_{i}}(\sigma,t).$$
Then, applying
(\ref{nilpotent_exponential}) and (\ref{variation_formula}),
\[\begin{array}{l} (t,\dot{\xi}_0^{\eps_1,\ldots,\eps_l}(t))=
\overrightarrow{{\rm exp}}\int^t_0 \Big(X^e
\\ +\sum_{a=1}^{N_0}u_{{\rm slow},a}^{\eps_1,\dots,\eps_{l-1}}(s)V^0_a
\\+\frac{1}{\eps_l}\sum_{a=1}^{N_0}u_{{\rm
osc},a}^{\eps_1,\dots,\eps_{l-1}}(s/\eps_l,s)V^0_a \Big){\rm
d}s\;{\gamma}_{\mathrm{ref}}^e(0)\\ =\overrightarrow{{\rm
exp}}\int^t_0\frac{1}{\eps_l}\sum_{a=1}^{N_0}u_{{\rm
osc},a}^{\eps_1,\dots,\eps_{l-1}}(s/\eps_l)V^0_a {\rm d}s\,\circ
\\
\overrightarrow{{\rm exp}}\int^t_0 \Big(X^e+\sum_{a=1}^{N_0}u_{{\rm
slow},a}^{\eps_1,\dots,\eps_{l-1}}(s)V^0_a  \\
+ \sum_{a=1}^{N_0} \int^s_0 \frac{1}{\eps_l}u_{{\rm
osc},a}^{\eps_1,\dots,\eps_{l-1}}(s_0/\eps_l){\rm
d}s_0[V^0_a,X^e] \\
-\sum_{a=1}^{N_0} \int^s_0 \frac{1}{\eps_l}\partial_2 u_{{\rm
osc},a}^{\eps_1,\dots,\eps_{l-1}}(s_0/\eps_l){\rm
d}s_0 V^0_a \\
- \sum_{b,c=1}^{N_0}\int^s_0\frac{1}{\eps_l}u_{{\rm
osc},b}^{\eps_1,\dots,\eps_{l-1}}(s_1/\eps_l)
\\ \int^{s_1}_0\frac{1}{\eps_l}u_{{\rm
osc},c}^{\eps_1,\dots,\eps_{l-1}}(s_0/\eps_l){\rm d}s_0{\rm
d}s_1W^0_{b,c}\Big) {\rm d}s \,
{\gamma}_{\mathrm{ref}}^e(0)\end{array}\] Denoting by $\mathscr{V}$
any vertical flow, we have
\[\begin{array}{l} (t,\dot{\xi}_0^{\eps_1,\ldots,\eps_l}(t))
 = {\mathscr{V}}\,\circ \;
\overrightarrow{{\rm exp}}\int^t_0 \Big(X^e \\
+\sum_{a=1}^{N_0} \lambda_a^{\eps_1,\dots,\eps_{l-1}}(s)V^{0}_a
\\ +\sum_{b,c=1, \,
b<c}^{N_0} \lambda_{b,c}^{\eps_1,\dots,\eps_{l-1}}(s)W^0_{b,c}  \\
+ \int^s_0 \frac{1}{\eps_l}\sum_{a=1}^{N_0}u_{{\rm
osc},a}^{\eps_1,\dots,\eps_{l-1}}(s_0/\eps_l,s){\rm
d}s_0[V^{0}_a,X^e]\\
-\sum_{a=1}^{N_0} \int^s_0 \frac{1}{\eps_l}\partial_2 u_{{\rm
osc},a}^{\eps_1,\dots,\eps_{l-1}}(s_0/\eps_l,s){\rm d}s_0 V^{0}_a
\\+\sum_{b,c=1}^{N_0}(\Lambda_T(s,u_{{\rm
osc},b}^{\eps_1,\dots,\eps_{l-1}},u_{{\rm
osc},c}^{\eps_1,\dots,\eps_{l-1}})-\\
\int^s_0\frac{1}{\eps_l}u_{{\rm
osc},b}^{\eps_1,\dots,\eps_{l-1}}(s_1/\eps_l,s)
 \int^{s_1}_0\frac{1}{\eps_l}u_{{\rm
osc},c}^{\eps_1,\dots,\eps_{l-1}}(s_0/\eps_l,s)\\{\rm d}s_0{\rm
d}s_1)W^0_{b,c} \Big) {\rm d}s\,
{\gamma}_{\mathrm{ref}}^e(0).\end{array}\]
Noticing that
\[
\begin{array}{l}
\sum_{a=1}^{N_0} \lambda_a^{\eps_1,\dots,\eps_{l-1}}(s) V^{0}_a
+\sum_{b,c=1}^{N_0} (\lambda_{b,c}^{\eps_1,\dots,\eps_{l-1}}(s)
W^{0}_{b,c}\\ =\sum_{\hat{a}=1}^{N_1}\left(u_{{\rm
slow},{\hat{a}}}^{\eps_1,\dots,\eps_{l-2}}(s)
+\frac{1}{\eps_{l-1}}u_{{\rm
osc},{\hat{a}}}^{\eps_1,\dots,\eps_{l-2}}(s/\eps_{l-1},s)\right)V^{1}_{\hat{a}}
\end{array}
\]
and applying iteratively the same computation as above, one ends up with
\[\begin{array}{l} (t,\dot{\xi}_0^{\eps_1,\ldots,\eps_l}(t))=\\
 \mathscr{V}\,\circ  \overrightarrow{{\rm exp}}\int^t_0
\left(X^e+\lambda_a(s) V^{l}_a+  \mathscr{T}(s)\right){\rm
d}s\; \gamma_{\mathrm{ref}}^e(0),
\end{array}\]
where $\mathscr{T}(s)$ is a sum of terms of the form $\zeta(s)V$ where
$V\in \mathscr{X}(\R\times TQ)$ is independent of the $\eps_j$,
while $\zeta\in\mathcal{C}^\infty(I,\R)$ 
is of one
of the following four types:\[\begin{array}{l} \zeta_1=\int^{s}_0
\frac{1}{\eps_{i}}u_{{\rm
osc},a}^{\eps_1,\dots,\eps_{i-1}}(s_0/\eps_{i},s){\rm d}s_0,\\
\zeta_2=-\int^{s}_0 \frac{1}{\eps_{i}}\partial_2 u_{{\rm
osc},a}^{\eps_1,\dots,\eps_{i-1}}(s_0/\eps_{i},s){\rm
d}s_0, \\
\zeta_3=\frac 1{1+\delta_{bc}}\Big[2\Lambda_T(s,u_{{\rm
osc},b}^{\eps_1,\dots,\eps_{i-1}},u_{{\rm
osc},c}^{\eps_1,\dots,\eps_{i-1}})\\
-
\left(\int_0^s\frac{1}{\eps_{i}}u_{{\rm
osc},b}^{\eps_1,\dots,\eps_{i-1}}(s_0/\eps_{i},s){\rm d}s_0\right)
\\ \left(\int_0^s\frac{1}{\eps_{i}}u_{{\rm
osc},c}^{\eps_1,\dots,\eps_{i-1}}(s_0/\eps_{i},s){\rm
d}s_0\right)\Big] ,\end{array}\] for $i=1,\ldots, l$, and
\[\begin{array}{l}
\zeta_4= -\left(\int_0^s\frac{1}{\eps_{j}}u_{{\rm
osc},b}^{\eps_1,\dots,\eps_{j-1}}(s_0/\eps_{j},s){\rm d}s_0\right)
\\
\left(\int_0^s\frac{1}{\eps_{i}}u_{{\rm
osc},c}^{\eps_1,\dots,\eps_{i-1}}(s_0/\eps_{i},s){\rm d}s_0\right),
\end{array}\]
 for $i>j$, $i,j\in\{1,\ldots, l\}$.

We are left to prove that every $\int_0^t\zeta(s){\rm d}s$ converges to zero uniformly
with respect to $t$ as $\eps$ goes to zero (Lemma~\ref{ConvIntegral}).

Applying Lemma~\ref{simple_fact} with
 $$f(\tau,s)=\int^{\tau}_0
u_{{\rm
osc},a}^{\eps_1,\dots,\eps_{i-1}}(s_0
,s){\rm
d}s_0$$
and $\hat\eps=\eps_i$
 leads to
 $$\int_0^t \zeta_1(s) {\rm
d}s\leq C \eps_i (\eps_1\cdots \eps_{i-2}\eps_{i-1}^2)^{-1}.$$
Similarly,
\[
\int_0^t \zeta_2(s) {\rm d}s \leq C  (\eps_1\cdots
\eps_{i-2}\eps_{i-1}^3)^{-1}\eps_i.\] Taking
\[\begin{array}{l}f(\tau,s)= \left(\int_0^\tau u_{{\rm
osc},b}^{\eps_1,\dots,\eps_{i-1}}(s_0,s){\rm d}s_0\right)
\\ \quad \left(\int_0^\tau u_{{\rm
osc},c}^{\eps_1,\dots,\eps_{i-1}}(s_0,s){\rm
d}s_0\right),\end{array}\] we obtain
\[ \int_0^t \zeta_3(s) {\rm
d}s \leq C  (\eps_1\cdots \eps_{i-2})^{-2}\eps_{i-1}^{-3}\eps_i.\]
Finally, with
\[\begin{array}{l}f(\tau,s)=\left(\int_0^{s/\eps_j}u_{{\rm
osc},b}^{\eps_1,\dots,\eps_{j-1}}(s_0,s){\rm d}s_0\right) \\ \quad
\left(\int_0^\tau u_{{\rm
osc},c}^{\eps_1,\dots,\eps_{i-1}}(s_0,s){\rm d}s_0\right),
\end{array}
\]
we have \[\int_0^t \zeta_4(s) {\rm d}s \leq C  (\eps_1\cdots
\eps_{j-2}\eps_{j-1})^{-1}(\eps_1\cdots
\eps_{i-2}\eps_{i-1}^2)^{-1}\eps_i.
\]
Hence, each $\zeta$ satisfies
$$
 \int_0^t\zeta(s) {\rm
d}s\leq  \frac{C\eps_i}{(\eps_1\cdots \eps_{i-2})^{2}\eps_{i-1}^{3}}
= \frac{C\eps_i}{\eps_{i-1}^{4}}\frac{\eps_{i-1}}{(\eps_1\cdots
\eps_{i-2})^{2}}
$$
and it is easy to prove by recurrence that $\eps_{i-1}(\eps_1\cdots
\eps_{i-2})^{-2}$ tends to zero as $\eps$ goes to zero. \end{proof}

\begin{remark}\label{H_quantification-Z}
As in Remark \ref{H_quantification}, the hypothesis
$\limsup_{\eps\rightarrow0}\eta_i(\eps)/\eps^4<\infty$ in the
statement of Theorem~\ref{thm-1-parameter-Z} can be weakened by
requiring that $\limsup_{\eps\rightarrow
0}\eta_i(\eps)/\eps^{3+a}<\infty$ with $a=\sqrt{3}-1\simeq 0.73$.
Indeed, with this choice of $a$,
\[
\frac{\eps_i^a }{\eps_1^2\eps_2^2 \cdots \eps_{i-1}^2}=\left(
\frac{\eps_i}{\eps_{i-1}^{3+a}} \right)^a \left(
\frac{\eps_{i-1}^a}{\eps_1^2 \cdots \eps_{i-2}^2}\right),\] so that
\[\begin{array}{l}
\frac{\eps_i }{\eps_1^2\cdots \eps_{i-2}^2\eps_{i-1}^{3}}=\\
\frac{\eps_i}{\eps_{i-1}^{3+a}}\left(
\frac{\eps_{i-1}}{\eps_{i-2}^{3+a}} \right)^a \left(
\frac{\eps_{i-2}}{\eps_{i-3}^{3+a}} \right)^a\cdots \left(
\frac{\eps_2}{\eps_{1}^{3+a}} \right)^a\eps_1^a.
\end{array}\]
Hence, each $ \int_0^t\zeta(s) {\rm d}s$ goes to zero uniformly with
respect to $t\in I$ as $\eps$ tends to zero.
\end{remark}

Under special assumptions the relations in
(\ref{to-be-quantified}) that have been quantified in Theorem
\ref{thm-1-parameter-Z} can be reduced up to $\lim_{\eps \to
0}\eta_i(\eps)/\eps^2=0$ as stated in the following corollary and
illustrated in the numerical simulation included in Section
\ref{SSimulation}. The required assumptions are stated in terms of the
number of steps in the algorithm and of the coefficients providing the parametrization of the reference trajectory.

\begin{corol}\label{Corol-reduction-to-eps2}
Let $\Sigma=(Q,\nabla,Y,\mathscr{Y},\mathbb{R}^k)$ be a FACCS. Let
$\mathscr{Z}_i$, $i\in \mathbb{N}_0$, be defined as in (\ref{setZ})
and assume that ${\rm span}_{\mathbb{R}}\mathscr{Z}_2(q)=T_qQ$ for
all $q\in Q$ and that for each $i\in \{0,1\}$ and each $Z\in
\mathscr{Z}_i$, $\langle Z\colon Z\rangle (q)\in {\rm
span}_{\mathbb{R}}(\mathscr{Z}_i(q))$.
 Fix
a reference trajectory $\gref\in \mathcal{C}^\infty(I,Q)$ such that
the coefficients $\lambda_a$ associated with $\gref$
as
in (\ref{Eq-First-Param-Case Z}) are constant
functions.
Construct $\xi^{\eps_1,\eps_2}$, $u^{\eps_1,\eps_2}$ and
$\hat\eta_i$ as above. If $\eta_2:(0,+\infty)\rightarrow(0,+\infty)$
satisfies $\lim_{\eps\rightarrow 0}\eta_2(\eps)/\eps^2=0$, then
$\eps\mapsto
u^{\eps,\eta_2(\eps)}$ is a
one\textendash{}parameter tracking control law for $\gref$.
\end{corol}

\begin{proof} The hypotheses of the corollary
and the expressions of the controls appearing in Theorem~\ref{ACCStrack}
guarantee that the
oscillatory controls are as follows:
\begin{align*} u_{{\rm osc},a}(\tau,t)&\in  {\rm span}_{\mathbb{R}}\{\varphi_{j_0}(\tau)\}, \\
u_{{\rm osc},a}^{\eps_1}(\tau,t) & \in {\rm
span}_{\mathbb{R}}\{\varphi_{j_0}(\tau),
\frac{1}{\eps_1}\varphi_{j_1}(t/\eps_1)\varphi_{j_2}(\tau)\},
\end{align*}
where $j_0,j_1,j_2$ vary in $\N$. 

Thus the integrals of all the terms $\mathscr{T}$ that appear in the
proof of Theorem~\ref{thm-1-parameter-Z}
only contain product of
trigonometric functions, more specifically cosines and sines
and
converge to zero
uniformly with respect to $t$ as $\eps$ goes to zero
if $\eta_2$
is
such that $\lim_{\eps\rightarrow
0}\eta_2(\eps)/\eps^2=0$.
 \end{proof}

\subsubsection{Numerical simulation: submarine}\label{SSimulation}

In this section, we illustrate the method to obtain a
one\textendash{}parameter control law in a concrete situation that
fulfills the assumptions in Corollary \ref{Corol-reduction-to-eps2}.
The algorithm described in Section \ref{SCaseZ} has been implemented
with Scilab.

For our example, we have the submarine presented in Section
\ref{SExSubmarine} whose dynamics are given by (\ref{eq-below}) and
(\ref{eq-KirchSubm}). We consider the same kind of inertia matrix as
in Section \ref{SExSubmarine}, taking
\[J_1=1, \quad J_3=3, \quad M_1=1,\quad M_2=2, \quad M_3=3.\]
We recall that in the case under consideration
Theorem \ref{ACCStrack} cannot be applied. However,
our method provides a one\textendash{}parameter control law that
solves the tracking problem.
The trajectory to be tracked is given by
\[r(t)=(-t,-t,-t), \quad A_{33}(t)=1, \quad t\in [0,1]\]
with initial condition $r(0)=(0,0,0)$, $A(0)={\rm Id}$, being ${\rm
Id}$ the identity $3\times 3$ matrix, $\Pi(0)=(0,0,0)$,
$P(0)=(-1,-2,-3)$. Thus there are degrees of freedom in the attitude
of the submarine, but the target position of the center of the
submarine is fully determined.

\begin{figure}[thpb]
      \begin{center}
\includegraphics[width=\columnwidth]{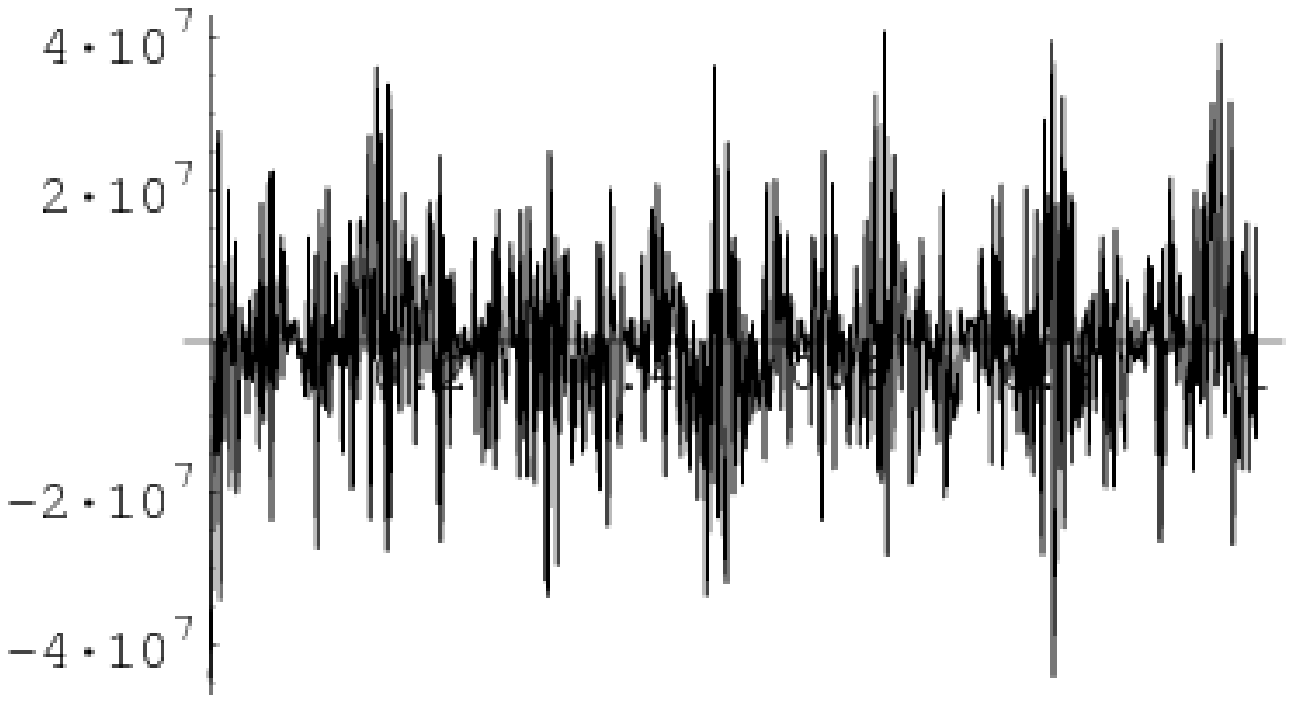}
\caption{The 
control law $u_1^{\eps,\eta_2(\eps)}$.
}\label{Fig-Control1}
\end{center}
\end{figure}

\begin{figure}[thpb]
      \centering
\includegraphics[width=\columnwidth]{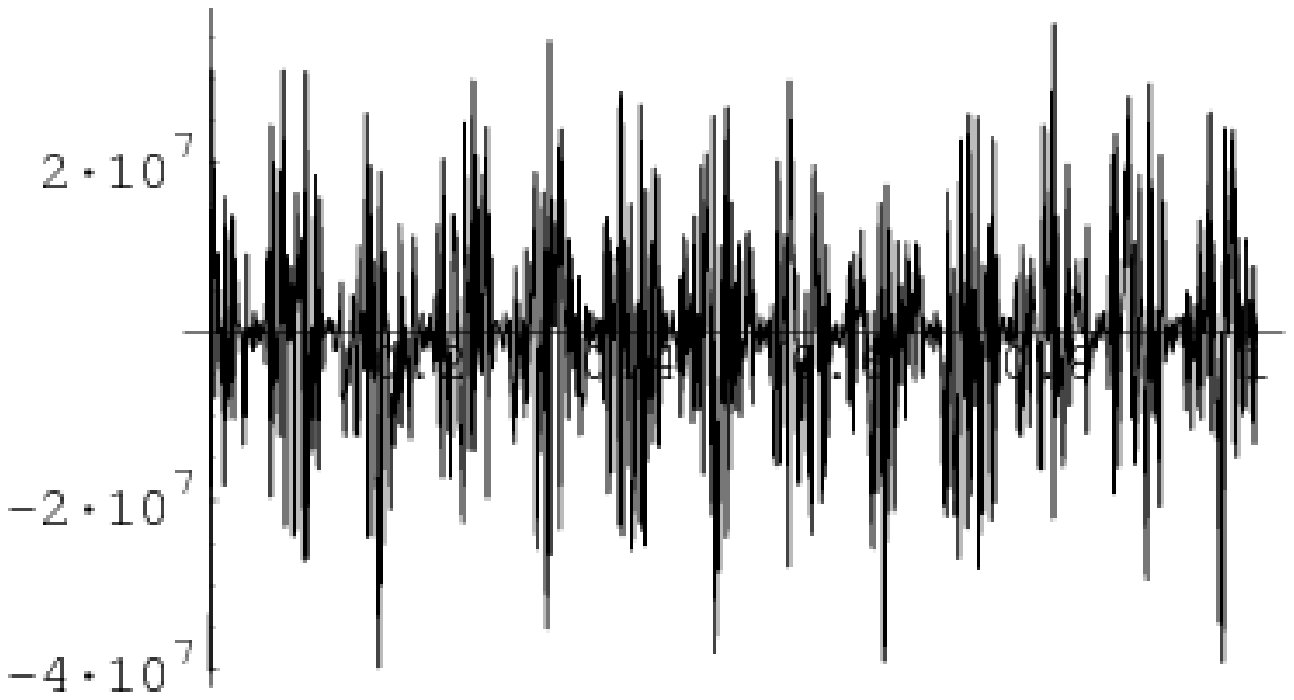}
\caption{The 
control law $u_2^{\eps,\eta_2(\eps)}$.
}\label{Fig-Control2}
\end{figure}
\begin{figure}[thpb]
      \centering
\includegraphics[width=\columnwidth]{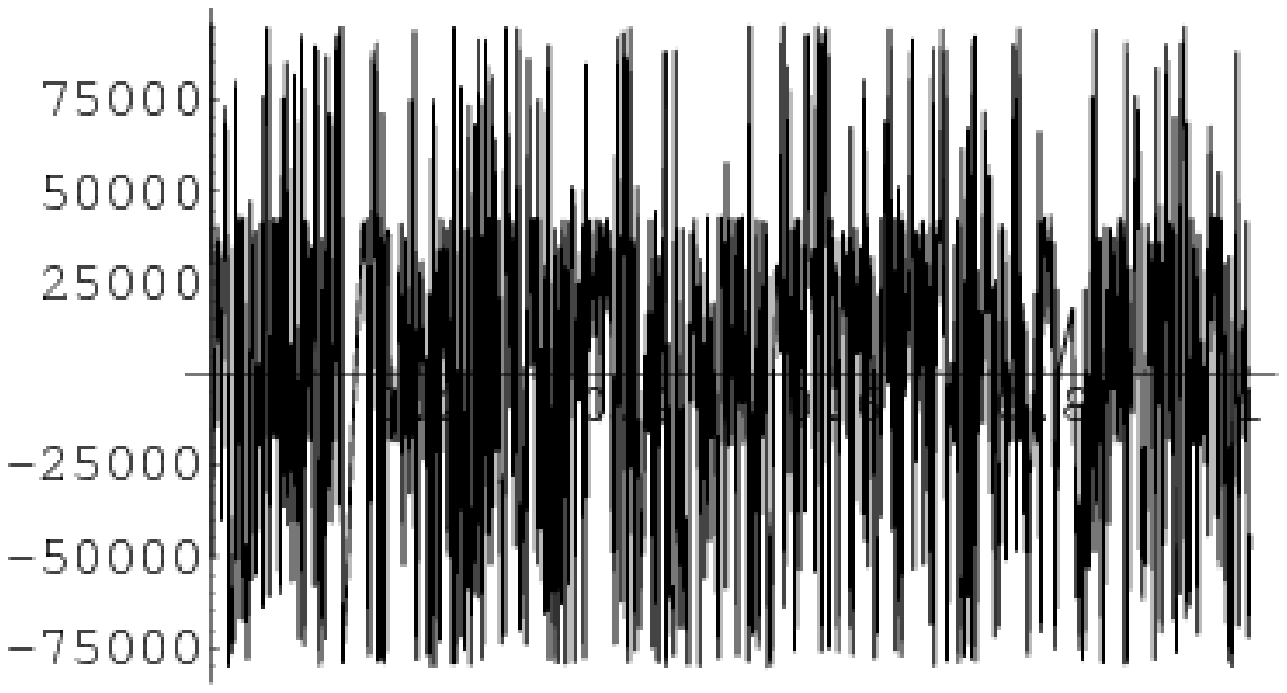}
\caption{The 
control law $u_3^{\eps,\eta_2(\eps)}$.
}\label{Fig-Control3}
\end{figure}

In this implementation we take $\eta_2(\eps)=\eps^{2.5}$ because the
considered reference trajectory satisfies the hypotheses of
Corollary~\ref{Corol-reduction-to-eps2}.

First, we compute the one\textendash{}parameter control laws
$u_1^{\eps_1,\eps_2}$, $u_2^{\eps_1,\eps_2}$ and
$u_3^{\eps_1,\eps_2}$ as described in the proof of
Theorem~\ref{thm-1-parameter-Z}. Then we fix $\eps=1/39\approx
0.0256$, so that $\eta_2(\eps)\approx 0.0001$. The corresponding
control laws are represented in Figures \ref{Fig-Control1},
\ref{Fig-Control2} and \ref{Fig-Control3}. By construction, the
controls are highly oscillatory.

\begin{figure}[thpb]
      \centering
\includegraphics[width=\columnwidth]{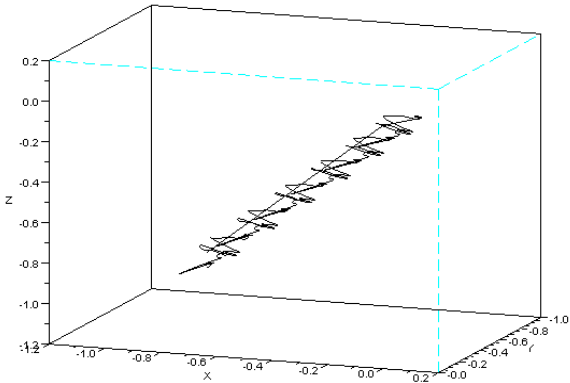}
\caption{Evolution of the position of the center of the submarine
with respect to time. The target trajectory is the
non\textendash{}oscillating curve.}\label{Fig-track-r}
\end{figure}
\begin{figure}[thpb]
      \centering
\includegraphics[width=\columnwidth]{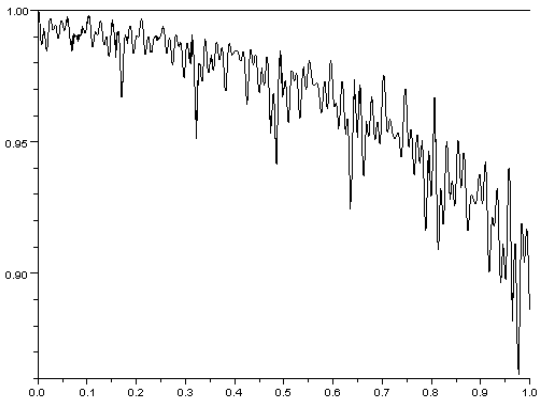}
\caption{Evolution of $A_{33}$ entry of the attitude matrix with
respect to time. The target trajectory is the
non\textendash{}oscillating curve.}\label{Fig-track-A33}
\end{figure}

Then we integrate the dynamics of the system using the numerical
integrator \textit{stiff} included in Scilab. As a result, Figures
\ref{Fig-track-r} and \ref{Fig-track-A33} show that the target
trajectory, corresponding with the non\textendash{}oscillating line,
is tracked by the oscillating curve. The error of the approximation,
computed by the supremum distance, is ${\rm d}((r_{{\rm
ref}}(t),A_{33,{\rm ref}}(t)),(r(t),A_{33}(t)))\approx 0.1903$.

\section{Conclusions}
The previously  known sufficient conditions for tracking were given
in terms of finite sets of vector fields, as reviewed in Section
\ref{STrack}. Here we have constructed a sequence of infinite family
of vector fields that defines a sequence of convex cones suitable
for characterizing trackability (Theorem \ref{GeneralTrack} and
Corollary \ref{CorolH}).
Different convex cones, (\ref{setK}) and (\ref{setH}), have been
considered. Under additional assumptions, using the cones in
(\ref{setH}) and a particular sequence of finite families of vector
fields, it is possible to recover the
sufficient conditions for
tracking already known in the literature \cite{2005BulloAndrewBook},
see Corollary \ref{CorolSym1}. However, our constructions not only
recover the previously known results,
but they also extend
them,
as shown in Section \ref{exps}.

The sequence of families of vector fields in Corollaries
\ref{CorolH} and \ref{CorolSym1} are also suitable for constructing
a one\textendash{}parameter tracking control law (Theorems
\ref{thm-1-parameter} and \ref{thm-1-parameter-Z}). It
remains as future work to generalize the construction of
one\textendash{}parameter tracking control laws
when the sets (\ref{setK}),
that include the closure, are considered.

Another future research line is the study of the complexity for the
control\textendash{}affine systems considered in this work.
Apart from tracking
non\textendash{}admissible trajectories, one could impose more requirements
on the solution to the
tracking problem, as for instance, to save energy.
The complexity
provides a good tool to formulate
this kind of problems and, so far, has been only
studied for control\textendash{}linear systems \cite{Gauthier,Jean}.

\section*{Acknowledgements}

The first author acknowledges the financial support of Comissionat
per a Universitats i Recerca del Departament d'Innovaci\'o,
Universitats i Empresa of Generalitat de Catalunya in the final
preparation of this paper.


\end{document}